\documentclass[12pt]{amsart}
\usepackage{amscd,amsmath,amsthm,amssymb,verbatim,enumerate}
\usepackage{color}
\usepackage{amsfonts,latexsym,amsthm,amssymb,amsmath,amscd,euscript}
\usepackage{tikz-cd}
\usepackage[margin = 2cm]{geometry}
\usepackage[shortlabels]{enumitem}
\usepackage[utf8]{inputenc}
\usepackage{hyperref}
\usepackage{cleveref}
\usepackage{arydshln}
\usepackage{tikz}
\usepackage{float}
\usepackage{graphicx}
\usepackage[ruled, vlined]{algorithm2e}

\numberwithin{equation}{section}
\newcommand{\nc}{\newcommand}

\nc{\R}{\mathbb R}
\nc{\C}{\mathbb C}
\nc{\F}{\mathcal F}
\nc{\Q}{\mathbb Q}
\nc{\Z}{\mathbb Z}
\nc{\N}{\mathbb N}
\nc{\V}{\mathcal{V}}
\nc{\J}{\mathcal{J}}
\nc{\Sasha}[1]{{\color{violet} \sf $\heartsuit$ [#1]}}
\nc{\Asli}[1]{{\color{red} \sf $\heartsuit$ [#1]}}
\nc{\Emanuela}[1]{{\color{green} \sf $\heartsuit$ [#1]}}
\nc{\Sudeshna}[1]{{\color{blue} \sf $\heartsuit$ [#1]}}
\nc{\Jennifer}[1]{{\color{magenta} \sf $\heartsuit$ [#1]}}

\nc{\red}{\textcolor{red}}
\nc{\green}{\textcolor{green}}
\nc{\blue}{\textcolor{blue}}

\newtheorem{theorem}{Theorem}[section]

\newtheorem{proposition}[theorem]{Proposition}
\newtheorem{corollary}[theorem]{Corollary}

\newtheorem{question}[theorem]{Question}

\theoremstyle{definition}
\newtheorem{definition}[theorem]{Definition}
\newtheorem{remark}[theorem]{Remark}
\newtheorem{example}[theorem]{Example}

\DeclareMathOperator{\des}{des}
\DeclareMathOperator{\height}{height}
\DeclareMathOperator{\ini}{in}
\DeclareMathOperator{\HS}{HS}
\DeclareMathOperator{\HF}{HF}

\DeclareMathOperator{\rank}{rank}
\DeclareMathOperator{\reg}{reg}
\DeclareMathOperator{\e}{e}
\DeclareMathOperator{\ainv}{a}
\def\sort{\operatorname{sort}}

\begin{document}

\title{Invariants of toric double determinantal rings}

\author[Biermann]{Jennifer Biermann}
\address{Department of Mathematics and Computer Science, Hobart and William Smith Colleges \\
Geneva, NY 14456}
\email{biermann@hws.edu}

\author[De Negri]{Emanuela De Negri}
\address{Dipartimento di Matematica, Universit\`a di Genova, Via Dodecaneso 35, 16146 Genova, Italy}
\email{denegri@dima.unige.it}

\author[Gasanova]{Oleksandra Gasanova}
\address{Faculty of Mathematics, University of Duisburg-Essen, Essen 45127, Germany}
\email{oleksandra.gasanova@uni-due.de}

\author[Musapasaoglu]{Asl\i\  Musapa\c{s}ao\u{g}lu}
\address{Sabanci University, Faculty of Engineering and Natural Sciences, Orta Mahalle\\
 Tuzla 34956 Istanbul, Turkey}
\email{atmusapasaoglu@sabanciuniv.edu}

\author[Roy]{Sudeshna Roy}
\address{Department of Mathematics, Indian Institute of Technology Gandhinagar, Palaj, Gandhinagar, 382055 Gujarat, India}
\email{sudeshna.roy@iitgn.ac.in}

\date{\today}

\begin{abstract}
We study a class of double determinantal ideals denoted  $I_{mn}^r$, which are generated by minors of size 2, and show that they are equal to the Hibi rings of certain finite distributive lattices. We compute the number of minimal generators of $I_{mn}^r$, as well as the multiplicity, regularity, a-invariant, Hilbert function, and $h$-polynomial of the ring $R/I_{mn}^r$, and we give a new proof of the dimension of $R/I_{mn}^r$.  We also characterize when the ring $R/I_{mn}^r$ is Gorenstein, thereby answering a question of Li in the toric case.  Finally, we give combinatorial descriptions of the facets of the Stanley-Reisner complex of the initial ideal of $I_{mn}^r$ with respect to a diagonal term order. 
\end{abstract}
\maketitle
%%%%%%%%%%%%%%%%%%%%%%%%%%%%%%%%%%%%%%%%%%%%%%%%%%%%%5

\section*{Introduction}

Determinantal rings and varieties are a central topic in commutative algebra and algebraic geometry. They have been studied from a variety of points of view and with methods coming from algebra, geometry, and combinatorics, see for example \cite{BCRV}.
Besides the classical determinantal varieties defined by minors of a generic matrix, other more general classes have been studied, for example ladder determinantal rings (see \cite{C}, \cite{GM}, \cite{G}), which are defined by special subregions of a generic matrix. 

Double determinantal ideals $I_{mn}^r$ are, by contrast, generated by minors, possibly of different size, of the vertical and of the horizontal concatenation of $r$ generic $m\times n$ matrices of indeterminates. They were introduced in \cite{L1} by Li as instances of Nakajima quiver varieties, but are also a natural generalization of classical determinantal rings and varieties, which can be obtained by setting $r=1$.

The starting point of our work is two papers. The first is \cite{FK}, where Klein and Fieldsteel prove, answering a conjecture of Li, that double determinantal varieties are normal, irreducible, and Cohen-Macaulay, and that their defining ideals have a Gr\"obner basis given by their natural generators. The second is \cite{CDS}, where the case $r=2$ is studied and a description of the simplicial complex associated to the initial ideal is given; moreover, a formula for the multiplicity and a bound for the regularity are given, and a conjecture on Gorensteinness is stated. 

In \cite{L}, Li introduced a more general class of ideals, the bipartite determinantal ideals of which double determinantal ideals are a subclass. In a later paper \cite{IL}, Illian and Li prove that the natural generators are a Gröbner basis with respect to any diagonal order. In \cite{L}, the author gives a description of the simplicial complex associated to the corresponding initial ideal. The complex turns out to be vertex decomposable, and thus shellable. Moreover, the author writes a combinatorial formula for the $h$-vector and for the multiplicity, by describing families of paths whose cardinality equals these invariants. But finding a closed formula for these cardinalities is very hard, in fact Li asks in \cite[Sections 5.3.2,5.3.3]{L} for a nice formula for the $h$-vector and multiplicity, and characterization of Gorensteinness. 

In this paper, we give answers to Li's questions in the toric case.
We study toric double determinantal rings by using combinatorial tools. This case has already been studied in \cite{BKMM}, where the authors give new proofs of some of the results in \cite{FK} and prove that double determinantal varieties are smooth. We use a different approach to study these rings.

Let $I_{mn}^r$ be the ideal generated by the $2$-minors of the horizontal and vertical concatenation of the generic $m\times n$ matrices $X_1,\ldots,X_r$, and let $R$ be the polynomial ring with coefficients in a field $K$, and indeterminates the $mnr$ entries of the matrices.

In Sections 1 and 2 of this paper we give background material on $K$-algebras generated by sortable sets of monomials, lattices, posets and Hibi rings.  Moreover we recall how a Hibi ring can be seen as a $K$-algebra generated by a sortable set of monomials and we introduce a finite distributive lattice that we call $L_{mnr}$.

In Section 3 we show that the toric double determinantal ring $R/I_{mn}^r$ is the Hibi ring of $L_{mnr}$ (\Cref{thm: toric presentation}). We also give an explicit description of a set of minimal generators of the ideal $I_{mn}^r$ (\Cref{thm: mingens}).

In Section 4 we exploit the fact that $R/I_{mn}^r$ is a Hibi ring to compute the number of minimal generators of $I_{mn}^r$, as well as the multiplicity, regularity, a-invariant and Hilbert function of the ring $R/I_{mn}^r$, and give a new proof of the dimension of $R/I_{mn}^r$ (\Cref{thm: invariants}).  As corollaries to this result we give two descriptions of the $h$-polynomial of $R/I_{mn}^r$.  Moreover, we prove the surprising result that if $\sigma$ is a permutation of $\{m,n,r\}$, then $R/I_{mn}^r\cong R/I_{\sigma(m)\sigma(n)}^{\sigma(r)}$ (\Cref{thm: isomorphism}). 
The formulation of $R/I_{mn}^r$ as a Hibi ring also gives a new simple proof of the known fact (see \cite{BKMM} and \cite{L}) that $R/I_{mn}^r$ is a domain.  Lastly in \Cref{thm: Gorenstein}, we are able to answer a question of Li by providing a characterization of when the ring $R/I_{mn}^r$ is Gorenstein.

Finally, in Section 5, we study the simplicial complex $\Delta$ associated to the initial ideal of $I_{mn}^r$ with respect to a diagonal order. We give two descriptions of the facets of $\Delta$, in \Cref{facets} and \Cref{NMR}, the first of which generalizes the one given in \cite[Theorem 3.2]{CDS}.

\section{Preliminaries}
\subsection{Toric rings defined by sorting relations}
\label{sect:sortable}

Let $s$ be a positive integer and let $S_s$ be the
$K$-vector space spanned by the
monomials of degree $s$ in the standard graded polynomial ring $S = K[t_1,\ldots,t_{\ell}]$. 

We define a linear map
$$\sort: S_s \times S_s \to S_s \times S_s,$$ 
that we call the \emph{sorting map.}

Given two monomials $u_1,u_2\in
S_s$, we write $u_1u_2 = t_{\ell_1}t_{\ell_2}\cdots t_{\ell_{2s}}$ with $1\le \ell_1\le \ell_2\le\ldots\le \ell_{2s} \le \ell$ and set
$$u_3=t_{\ell_1}t_{\ell_3}\cdots t_{\ell_{2s-1}},\quad u_4= t_{\ell_2}t_{\ell_4}\cdots t_{\ell_{2s}}.$$ Then we set $\sort(u_1, u_2)=(u_3,u_4).$ 
Note that the sorting depends on the order of the variables.

\begin{definition}
\label{def: sorting} 
\begin{enumerate}
\item The pair $(u_3,u_4)=\sort(u_1, u_2)$ is called the \emph{sorting} of $(u_1, u_2)$. 
\item A pair $(u_1, u_2)$ is \emph{sorted} if $\sort(u_1, u_2)=(u_1, u_2)$. 
\item A subset $A \subseteq S_s$ of monomials is called \emph{sortable} if
$\sort(A \times A) \subseteq A \times A$. 
\end{enumerate}
\end{definition}
Since $\sort(u_1,u_2)=\sort (u_2,u_1)$, we will sometimes refer to unordered pairs and say that $\{u_1,u_2\}$ is sorted if either $(u_1, u_2)$ or $(u_2, u_1)$ is sorted.

Let $A\subset S_s$ be a sortable set of monomials and let $K[A]$ be the $K$-algebra
generated by $A$.  Let $R=K[x_u\mid u\in A]$ and let $\phi : R \to K[A]$ be
the $K$-algebra homomorphism defined by $x_u\mapsto u$ for all $u\in A$. Then $K[A] \cong R/I_A$ where $I_A = \ker(\phi)$.
\begin{theorem}
[{\cite[Theorem 14.2]{St}}]
\label{thm:sortingGB}
Let $A$ be a sortable set of monomials and let
$$G = \left\{x_{u_1}x_{u_2}-x_{u_3}x_{u_4}: \{u_1, u_2\} \text{ unsorted}, (u_3, u_4) = \sort(u_1, u_2)\right\} \subset R.$$
Then $I_A=(G)$ and moreover there exists a monomial order $\tau$ on $R$ with respect to which $G$ is a reduced Gröbner basis. 
\end{theorem}
The elements of $G$ in \Cref{thm:sortingGB} are called \emph{sorting relations}.
\begin{remark}
\label{rem: minimal sort}
\begin{enumerate}
    \item 
In \cite{St} Sturmfels constructs a monomial order and proves {\cite[Theorem 14.2]{St}} with respect to it. However, the proof only relies on the fact that this is a sorting order, that is, it is a monomial order $\tau$ which fixes unsorted pairs as leading terms. In other words, $\tau$ satisfies $\ini_{\tau}(x_{u_1}x_{u_2}-x_{u_3}x_{u_4})=x_{u_1}x_{u_2}$ for all $\{u_1, u_2\}$ unsorted, with $(u_3, u_4) = \sort(u_1, u_2)$. 

\noindent We want to emphasize that the statement of \cite[Theorem 14.2]{St} holds for \emph{any} sorting order.
     
\item The generating set $G$ in \Cref{thm:sortingGB} is minimal. On the contrary, suppose that $f,f_1,\ldots,f_q$ are distinct elements of $G$ and that $f=\varepsilon_1f_1+\cdots +\varepsilon_qf_q$. Then $\varepsilon_1,\ldots, \varepsilon_q\in \{1,-1\}$, since all these elements have  degree $2$. 
But the monomial corresponding to the unsorted pair in $f$ is different from all the terms appearing in $f_1,\ldots, f_q$, which gives a contradiction. 

\end{enumerate}
\end{remark}

\subsection{Posets, lattices and Hibi rings} We review the definitions and basic properties of posets and lattices. For further reading, we recommend \cite{B} and \cite{Stn}. 
A \emph{partially ordered set} $P$ (or \emph{poset}, for short) is a set, together with a binary relation $\preceq$ (or $\preceq_P$ when there is a possibility of confusion), 
satisfying the axioms of reflexivity, antisymmetry, and transitivity. In this paper all posets will be finite.

For $u_1, u_2 \in P$, we will say that $u_1$ and $u_2$ are \emph{comparable} (denoted by $u_1\sim u_2$) if either  $u_1\preceq u_2$ or $u_2\preceq u_1$, otherwise we will say that $u_1$ and $u_2$ are \emph{incomparable}. A \emph{chain} of $P$ is a subset $C\subseteq P$ such that every pair of elements in this subset is comparable. By $|C|$ we will denote the number of elements in $C$, while the \emph{length} of $C$ is defined as $|C|-1$. The \emph{rank} of $P$ is the length of its longest chain. If every maximal chain of a finite poset $P$ has the same length, then we say that $P$ is \textit{pure}. The \textit{width} of a poset is the smallest number of disjoint chains needed to cover $P$.

For $u_1, u_2 \in P$, a \emph{join} of $u_1$ and $u_2$ is an element $u\in P$ such that $u\succeq u_1$ and $u\succeq u_2$, and such that any other element $u'$ with the same property satisfies $u'\succeq u$. If a join of $u_1$ and $u_2$ exists, it is clearly unique and is denoted by $u_1\vee u_2$.  Dually, one can define the \textit{meet} $u_1\wedge u_2$, when it exists. A \textit{lattice}
is a poset for which every pair of elements has a meet and a join. A lattice is called \emph{distributive} if the operations $\vee$ and $\wedge$ distribute over each other.

An \emph{ideal} of a poset $P$ is a subset $I$ of $P$ such that if $u_1 \in I$ and $u_2\preceq u_1$, then $u_2 \in I$. In other words, an ideal is a downward closed subset of $P$. The set $\mathcal{J}(P)$ of ideals of $P$ is a poset, ordered by inclusion. Since the union and intersection of ideals is again an ideal, it follows that $\mathcal{J}(P)$ is a lattice with the join of two ideals being their union, and the meet being their intersection. From the well-known distributivity of set union and intersection over
one another, it follows that $\mathcal{J}(P)$ is in fact a distributive lattice. 
Birkhoff's fundamental theorem for finite distributive lattices states that the converse is true when the lattice is finite (see for instance \cite[Theorem 3.4.1]{Stn}). 

\begin{theorem}(Birkhoff)
Let $L$ be a finite distributive lattice. Then there exists a finite poset $P$ such that $L\cong \J(P)$. 
\end{theorem}
\begin{definition}
    Let $L$ be a finite lattice. The \emph{Hibi ring} of $L$, denoted $K[L]$, is the quotient ring $K[x_u\mid u\in L]/(x_{u_1}x_{u_2}-x_{u_1\wedge u_2}x_{u_1\vee u_2}\mid u_1,u_2\in L)$.  
    The generators of this presentation ideal of $K[L]$ are called \emph{Hibi relations}.
\end{definition}
Clearly, only incomparable unordered pairs give non-trivial Hibi relations, thus  $$K[L]=K[x_u\mid u\in L]/(x_{u_1}x_{u_2}-x_{u_1\wedge u_2}x_{u_1\vee u_2}\mid \{u_1,u_2\}\subseteq L, u_1\nsim u_2).$$
These rings were introduced by Hibi in \cite{Hi}, where he proves that $K[L]$ is toric if and only if $L$ is distributive, and in this case $K[L]$ is a normal Cohen-Macaulay domain. All Hibi rings considered in this section will be defined by distributive lattices. A lot of invariants of Hibi rings can be read off directly from $L$ and $P$, where $L=\mathcal{J}(P)$, which we will use in the next section. For more information, the reader may consult the survey \cite{E}, and the papers \cite{EHM}, \cite{Hi} and \cite{HiHi}. The next theorem and remark are similar to those in \Cref{sect:sortable}.

\begin{theorem}[{\cite[Theorem 10.1.3]{HH}}] 
\label{thm:hibiGB}
Let $$G=\{x_{u_1}x_{u_2}-x_{u_1\wedge u_2}x_{u_1\vee u_2}\mid \{u_1,u_2\}\subseteq L, u_1\nsim u_2\}.$$
Then there exists a monomial order $\tau$ with respect to which $G$ is a reduced Gröbner basis.
\end{theorem} 

\begin{remark}
\label{rem: hibiorder}
\begin{enumerate}
    \item One can show that the above theorem holds for any monomial order $\tau$ which fixes incomparable pairs as leading terms, in other words, such that $\ini_{\tau}(x_{u_1}x_{u_2}-x_{u_1\wedge u_2}x_{u_1\vee u_2})=x_{u_1}x_{u_2}$ for all $\{u_1, u_2\}\subseteq L$ such that $u_1\nsim u_2$.  Any such $\tau$ is called a \emph{Hibi order}.
\item The generating set $G$ in \Cref{thm:hibiGB} is minimal. The proof is analogous to that of \Cref{rem: minimal sort}(2).

\end{enumerate}
\end{remark}

\section{Hibi rings as toric rings generated by sortable sets of monomials}
There are several known ways to view a toric Hibi ring as a semigroup ring, see \cite{Hi}.
%The similarities between \Cref{thm:sortingGB} and \Cref{thm:hibiGB}
We are interested in one specific way of doing this, which will turn out to be useful in studying toric double determinantal ideals. 

\begin{theorem}[{\cite[Theorem 5.1]{GN}}]
\label{thm: hibisortable}
 The Hibi ring of a finite distributive lattice $L$ can be realized as a $K$-algebra  generated by a sortable set of monomials indexed by the elements of $L$. Moreover, for $I, J \in L$  we have $\sort(u_I,u_J)=(u_{I\wedge J}, u_{I \vee J})$.
\end{theorem}
We outline the proof of {\cite[Theorem 5.1]{GN}}, since we are mostly interested in \emph{how} one obtains a monomial associated to an element of $L$.

\noindent \textit{Proof outline.} 
Let $P$ be a poset such that $L=\J(P)$, and let $s$  be the width of $P$. Fix a minimal chain decomposition of $P$, that is, disjoint chains $A_1, A_2,\ldots, A_s$ such that their set-theoretical union equals the underlying set of $P$. To each $I\in L$ we associate a vector $v_I\in  \mathbb{N}^s$, where $(v_I)_i=|I\cap A_i|+1$ for all $i=1,\ldots,s$. Note that the authors of {\cite[Theorem 5.1]{GN}} define this vector slightly differently; for our purposes it is more convenient to define it this way.

We consider the polynomial ring $K[x^{(1)}_{1},\ldots, x^{(1)}_{|A_1|+1}, \dots, x^{(s)}_{1},\ldots, x^{(s)}_{|A_s|+1}]$, and we associated to  $I$ the monomial $u_I=x^{(1)}_{(v_I)_1}\cdots x^{(s)}_{(v_I)_s}$. This assignment is injective, in other words, monomials $u_I$ are indexed by $I\in L$. Indeed, since every $I\in L$ is a downward closed subset of $P$, it contains exactly the $(v_I)_i-1$ \emph{smallest} elements of $A_i$ for every $i=1,\ldots, s$. Hence, every $v_I$ uniquely identifies $I$. This also implies that $v_{I\cap J}=\min\{v_I,v_J\}$ and $v_{I\cup J}=\max\{v_I,v_J\}$, where $\min$ and $\max$ are taken coordinate-wise. It is then not hard to see that the set of monomials $A=\{u_I\mid I\in \J(P)\}$ is sortable if one sorts the variables firstly by upper index and secondly by lower index. Indeed,
\begin{align*}
&\sort(u_I,\, u_J)= \sort(x^{(1)}_{(v_I)_1}\cdots x^{(s)}_{(v_I)_s},\ x^{(1)}_{(v_J)_1}\cdots x^{(s)}_{(v_J)_s})\\
&=(x^{(1)}_{\min\{(v_I)_1, (v_J)_1\}}\cdots x^{(s)}_{\min\{(v_I)_s, (v_J)_s\}},\ x^{(1)}_{\max\{(v_I)_1, (v_J)_1\}}\cdots x^{(s)}_{\max\{(v_I)_s, (v_J)_s\}})\\ &= (u_{I\cap J},\, u_{I\cup J}).
\end{align*}

 Let $R=K[x_I\mid I\in L]$ and let $\phi : R \to K[A]$ be
the $K$-algebra homomorphism defined by $x_I\mapsto u_I$ for all $I\in L$.
Since $A$ is a sortable set of monomials, the ideal $I_A=\ker\phi$ is generated by the sorting relations (\Cref{thm:sortingGB}). From the equality above, it is not hard to see that every sorting relation is a Hibi relation and vice versa (in particular, $\{u_I,u_J\}$ is unsorted if and only if $I\nsim J$). Hence $K[A]\cong R/I_A\cong K[L]$.

\bigskip

%%%%%%%%%%

Below we introduce a lattice that will be useful in the next section (see \Cref{fig: parallelepiped}).
\begin{definition}
\label{def:Lmnr}
  Given positive integers $m$, $n$, $r$, we denote by $P_{mnr}$ the poset consisting of $3$ disjoint chains $A_1$, $A_2$ and $A_3$ such that $|A_1|=m-1$, $|A_2|=n-1$ and $|A_3|=r-1$.  Moreover we set $L_{mnr}=\mathcal{J}(P_{mnr})$.
\end{definition}

In the next example, we consider the Hibi ring associated to the lattice just defined and, following the proof above, we show how it can be seen as a semigroup ring generated by a sortable set of monomials.
\begin{example}
\label{ex: hibisortable}
 We use the notation of \Cref{def:Lmnr}. Clearly, the width of $P_{mnr}$ is $s=3$, hence all the vectors have $3$ coordinates, and the associated monomials are in $K[x^{(1)}_{1},\ldots, x^{(1)}_{m}, x^{(2)}_{1},\ldots, x^{(2)}_{n}, x^{(3)}_{1},\ldots, x^{(3)}_{r}]$.
 .  
To each $I\in L_{mnr}$ we associate the vector $v_I=(\red{i},\green{j},\blue{k}) \in  \mathbb{N}^3$, where $i=|I\cap \red{A_1}|+1$, $j=|I\cap \green{A_2}|+1$, $k=|I\cap \blue{A_3}|+1$, and  the monomial $u_I=x_i^{(1)}x_j^{(2)}x_k^{(3)}$. Since we only have $3$ blocks of variables in our specific case, we  relabel $x_{*}^{(1)}\rightarrow x_{*}$, $x_{*}^{(2)}\rightarrow y_{*}$, $x_{*}^{(3)}\rightarrow z_{*}$. Hence if $v_I=(i,j,k)$, we will set  $u_I=x_iy_jz_k$. Note that the assignment $I\mapsto u_I$ is a bijection from $L_{mnr}$ to 
$$A_{mnr}=\{x_iy_jz_k\mid 1\le i\le m, 1\le j\le n,1\le k\le r\}.$$ Clearly, the set $A_{mnr}$ is sortable if one considers the natural order of variables:
\begin{align*}
\sort(u_I,\, u_J) &= \sort(x_{i_1}y_{j_1}z_{k_1},\ x_{i_2}y_{j_2}z_{k_2})\\
&=(x_{\min\{i_1,i_2\}}y_{\min\{j_1,j_2\}}z_{\min\{k_1,k_2\}},x_{\max\{i_1,i_2\}}y_{\max\{j_1,j_2\}}z_{\max\{k_1,k_2\}}\ )\\ &= (u_{I\cap J},\, u_{I\cup J}).
\end{align*}
Following the proof of \Cref{thm: hibisortable}, we conclude that $K[A_{mnr}]\cong K[L_{mnr}]$.
\iffalse
\begin{figure}[H]
\centering
\begin{tikzpicture}[scale=0.8]

\draw[thick,red] (-12,0)--(-12,2);
\filldraw[red] (-12,0) circle (2pt);
\filldraw[red] (-12,2) circle (2pt);
\node [thin, left] at (-12,0) {$\red{1}$};
\node [thin, left] at (-12,2) {$\red{2}$};
\node [thin, red] at (-12,-1) {$A_1$};

\filldraw[green] (-10,0) circle (2pt);
\node [thin, left] at (-10,0) {$\green{1}$};
\node [thin, green] at (-10,-1) {$A_2$};

\draw[thick,blue] (-8,0)--(-8,4);
\filldraw[blue] (-8,0) circle (2pt);
\filldraw[blue] (-8,2) circle (2pt);
\filldraw[blue] (-8,4) circle (2pt);
\node [thin, left] at (-8,0) {$\blue{1}$};
\node [thin, left] at (-8,2) {$\blue{2}$};
\node [thin, left] at (-8,4) {$\blue{3}$};
\node [thin, blue] at (-8,-1) {$A_3$};

\foreach \x in {0,...,2}{
\foreach \y in {0,...,1}{
\foreach \z in {0,...,3}{
\node [thin] at (-2*\x+2*\z,2*\x+2*\y+2*\z) {$\red{x_{\x}}\green{y_{\y}}\blue{z_{\z}}$};
\ifthenelse{\z > 0}
{\draw[thick] (-2*\x+2*\z-1.8,2*\x+2*\y+2*\z-1.8)--(-2*\x+2*\z-0.2,2*\x+2*\y+2*\z-0.2);}{}
\ifthenelse{\y > 0}
{\draw[thick] (-2*\x+2*\z,2*\x+2*\y+2*\z-1.8)--(-2*\x+2*\z,2*\x+2*\y+2*\z-0.2);}{}
\ifthenelse{\x > 0}
{\draw[thick] (-2*\x+2*\z+1.8,2*\x+2*\y+2*\z-1.8)--(-2*\x+2*\z+0.2,2*\x+2*\y+2*\z-0.2);}{}
}}}

\end{tikzpicture}
\caption{$P_{mnr}$ and $L_{mnr}$ for $\red{m}=3$, $\green{n}=2$, $\blue{r}=4$}
\label{fig: parallelepiped}
\end{figure}
\fi

\begin{figure}[H]
\centering
\begin{tikzpicture}[scale=0.8]

\draw[thick,red] (-12,0)--(-12,2);
\filldraw[red] (-12,0) circle (2pt);
\filldraw[red] (-12,2) circle (2pt);
\node [thin, left] at (-12,0) {$\red{1}$};
\node [thin, left] at (-12,2) {$\red{2}$};
\node [thin, red] at (-12,-1) {$A_1$};

\filldraw[green] (-10,0) circle (2pt);
\node [thin, left] at (-10,0) {$\green{1}$};
\node [thin, green] at (-10,-1) {$A_2$};

\draw[thick,blue] (-8,0)--(-8,4);
\filldraw[blue] (-8,0) circle (2pt);
\filldraw[blue] (-8,2) circle (2pt);
\filldraw[blue] (-8,4) circle (2pt);
\node [thin, left] at (-8,0) {$\blue{1}$};
\node [thin, left] at (-8,2) {$\blue{2}$};
\node [thin, left] at (-8,4) {$\blue{3}$};
\node [thin, blue] at (-8,-1) {$A_3$};

\foreach \x in {1,...,3}{
\foreach \y in {1,...,2}{
\foreach \z in {1,...,4}{
\node [thin] at (-2*\x+2*\z,2*\x+2*\y+2*\z-6) {$\red{x_{\x}}\green{y_{\y}}\blue{z_{\z}}$};
\ifthenelse{\z > 1}
{\draw[thick] (-2*\x+2*\z-1.8,2*\x+2*\y+2*\z-1.8-6)--(-2*\x+2*\z-0.2,2*\x+2*\y+2*\z-0.2-6);}{}
\ifthenelse{\y > 1}
{\draw[thick] (-2*\x+2*\z,2*\x+2*\y+2*\z-1.8-6)--(-2*\x+2*\z,2*\x+2*\y+2*\z-0.2-6);}{}
\ifthenelse{\x > 1}
{\draw[thick] (-2*\x+2*\z+1.8,2*\x+2*\y+2*\z-1.8-6)--(-2*\x+2*\z+0.2,2*\x+2*\y+2*\z-0.2-6);}{}
}}}

\end{tikzpicture}
\caption{$P_{mnr}$ and $L_{mnr}$ for $\red{m}=3$, $\green{n}=2$, $\blue{r}=4$}
\label{fig: parallelepiped}
\end{figure}
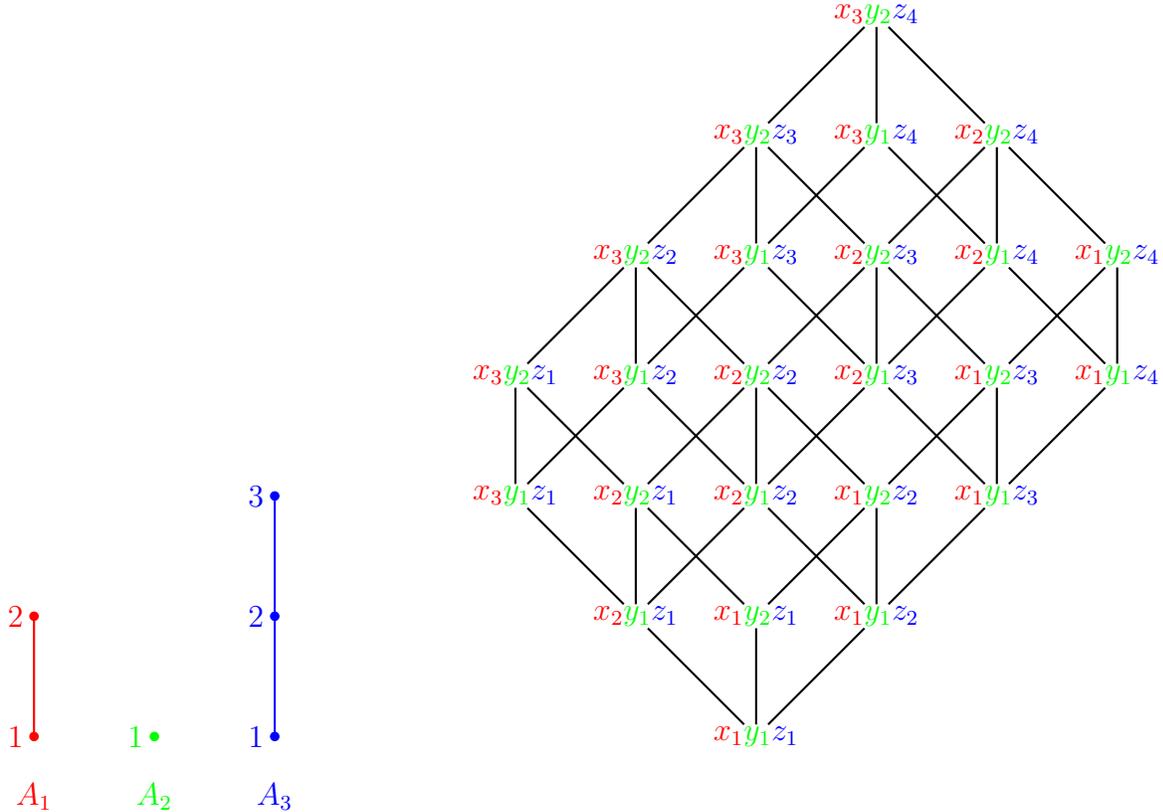
\end{example}

\section{Toric double determinantal rings as toric Hibi rings}
\subsection{Double determinantal rings}
Let $K$ be a field. Let $m,n,r$ be positive integers, let $X_k=(x_{ij}^k)$, with $k=1,\ldots,r$, be $m\times n$ 
matrices of distinct indeterminates, 
 and let $R = K[x_{ij}^k\,|\, 1\le i\le m,\ 1\le j\le n, \ 1 \le k \le r]$. 
Consider 
 $$H=(X_1\cdots X_r)  \quad \text{ and } \quad V =\left(\begin{matrix} X_1 \\ \vdots \\ X_r \end{matrix}\right)$$  
 the horizontal and vertical concatenation of the matrices $X_1,\ldots, X_r$. 
 The ideal $I_{mn}^r$  generated by the $t_h$-minors of $H$ and by the $t_v$-minors of $V$ is called a double determinantal ideal.
 In \cite{L}, a more general class of ideals is introduced, the bipartite determinantal ideals. By \cite[Corollary 3.3 and Corollary 4.5]{FK} and \cite[Corollary 1.3]{L} one has:
\begin{proposition}
\label{prop: CM}
The ring $R/I_{mn}^r$ is a Cohen-Macaulay domain for all $m,n,r\ge 1$. Assuming without loss of generality that $t_h\leq t_v$, one has
    \begin{align*}
    \height(I_{mn}^r)& =(m-t_h+1)(n-t_h+1)+(m-t_h+1)n(r-1)\\
    &+(n-t_v+1)\sum_{q=2}^r(t_h-1)-\max\{0, t_v-(t_h-1)(q-1) -1\}. 
    \end{align*}
\end{proposition}
\medskip
We study the case of $2$-minors. From now on, we always assume $t_h=t_v=2$ and denote by $I_{mn}^r$ the ideal generated by the $2$-minors of $H$ and $V$, unless stated otherwise. By \Cref{prop: CM}, we have 
\begin{equation} 
\label{eqn:dim}
\begin{split}
\dim R/I_{mn}^r&=mnr-\height(I_{mn}^r)\\
&=mnr-[(m-1)(n-1)+(m-1)n(r-1)+(n-1)\sum_{q=2}^r1]\\
&=mnr-[(m-1)(n-1)+(m-1)n(r-1)+(n-1)(r-1)]\\
&=m+n+r-2.
\end{split}
\end{equation}
It is easy to see that the natural generators of $I_{mn}^r$ are not minimal. 
\begin{example}
    Let $m=n=r=2$. Then 
    $$H=\begin{pmatrix} x_{11}^{1} &x_{12}^{1}& x_{11}^{2} &x_{12}^{2} \\[1ex]
    x_{21}^{1} & x_{22}^{1} & x_{21}^{2} & x_{22}^{2}
    \end{pmatrix}, \qquad
    V=\begin{pmatrix} x_{11}^{1} &x_{12}^{1} \\[1ex]
    x_{21}^{1} & x_{22}^{1} \\[1ex]
    x_{11}^{2} &x_{12}^{2} \\[1ex]
    x_{21}^{2} & x_{22}^{2}
    \end{pmatrix}.$$
    Thus $m_1=x_{11}^{1}x_{22}^{2}- x_{21}^{1}x_{12}^{2}$, $m_2=x_{12}^{1} x_{21}^{2}-x_{22}^{1}x_{11}^{2}$, $m_3=x_{11}^{1} x_{22}^{2}-x_{21}^{2}x_{12}^{1}$, $m_4=x_{21}^{1}x_{12}^{2}-  x_{11}^{2}x_{22}^{1}$ are generators of $I_{mn}^r$, since $m_1,m_2$ are minors in $H$ and $m_3,m_4$ are minors in $V$. But a direct calculation shows that $m_1+m_4=m_2+m_3$. 
\end{example}
We will give a description of a minimal set of generators in \Cref{thm: mingens} and a formula for the minimal number of generators in \Cref{thm: invariants}.

\subsection{Toric double determinantal rings as toric Hibi rings}

The goal of this subsection is to prove that for any positive integers $m$, $n$, $r$ one has $R/I_{mn}^r\cong K[L_{mnr}]$, where, as before, $R = K[x_{ij}^k\,|\, 1\le i\le m,\ 1\le j\le n, \ 1 \le k \le r]$, and $L_{mnr}$ is the lattice defined in \Cref{def:Lmnr}. First of all, from \Cref{thm: hibisortable} and \Cref{ex: hibisortable} we know that $K[L_{mnr}]\cong K[A_{mnr}]$, where $A_{mnr}=\{x_iy_jz_k\mid 1\le i\le m, \ 1\le j\le n, \ 1\le k\le r\}$. Hence our goal is to show that $R/I_{mn}^r\cong K[A_{mnr}]$. 

Let $\phi: R\to K[A_{mnr}]$ be the epimorphism given by $\phi(x_{ij}^k)=x_iy_jz_k$ and let $I_{A_{mnr}}=\ker\phi$. We have to prove that $I_{A_{mnr}}=I_{mn}^{r}$. We have a concrete description of a generating set of $I_{mn}^r$ (all the $2$-minors of $H$ and $V$), and the next theorem provides an equally concrete description of a generating set of $I_{A_{mnr}}$.

\begin{theorem}
\label{thm: mingens}
 Let 
 \begin{align*}
   \mathcal M&=\{x_{ij_2}^{k_1}x_{ij_1}^{k_2}-x_{ij_1}^{k_1}x_{ij_2}^{k_2}\mid 1\le i\le m, \ 1\le j_1<j_2\le n, \ 1\le k_1<k_2\le r\},\\[1ex]
   \mathcal N&=\{x_{i_2j}^{k_1}x_{i_1j}^{k_2}-x_{i_1j}^{k_1}x_{i_2j}^{k_2}\mid 1\le i_1< i_2\le m, \  1\le j\le n, \ 1\le k_1<k_2\le r\},\\[1ex]
   \mathcal R&=\{x_{i_1j_2}^{k}x_{i_2j_1}^{k}-x_{i_1j_1}^{k}x_{i_2j_2}^{k}\mid 1\le i_1< i_2\le m, \  1\le j_1<j_2\le n, \ 1\le k\le r\},\\[1ex] 
   \mathcal T&=\{x_{i_1j_2}^{k_1}x_{i_2j_1}^{k_2}-x_{i_1j_1}^{k_1}x_{i_2j_2}^{k_2}, \ x_{i_2j_1}^{k_1}x_{i_1j_2}^{k_2}-x_{i_1j_1}^{k_1}x_{i_2j_2}^{k_2},\ x_{i_2j_2}^{k_1}x_{i_1j_1}^{k_2}-x_{i_1j_1}^{k_1}x_{i_2j_2}^{k_2}\mid \\& 1\le i_1< i_2\le m, \ 1\le j_1<j_2\le n, \ 1\le k_1<k_2\le r\}.
    \end{align*}
Then $I_{A_{mnr}}$ is minimally generated by $\mathcal M\cup \mathcal N\cup\mathcal R\cup\mathcal T$. 
\end{theorem}
\begin{proof}
Since  $A_{mnr}$ is a sortable set of monomials (see \Cref{ex: hibisortable}), by \Cref{thm:sortingGB} and \Cref{rem: minimal sort}, we know that $I_{A_{mnr}}$ is minimally generated by the set $G$ of all the sorting relations (which can also be interpreted as Hibi relations). We will show that $G=\mathcal M\cup \mathcal N\cup\mathcal R\cup\mathcal T$. Clearly, $\mathcal M\cup \mathcal N\cup\mathcal R\cup\mathcal T\subseteq G$. For the other inclusion, let $\{u_1, u_2\}$ be an unsorted pair of monomials from $A_{mnr}$. Then $u_1u_2$ can be written as
$x_{i_1}x_{i_2}y_{j_1}y_{j_2}z_{k_1}z_{k_2}$ with $1\le i_1\le i_2\le m$, $1\le j_1\le j_2\le n$, $1\le k_1\le k_2\le r$. The sorting operation then returns $x_{i_1}y_{j_1}z_{k_1}$ and $x_{i_2}y_{j_2}z_{k_2}$. First of all we note that if $i_1=i_2$, $j_1=j_2$ and $k_1=k_2$, then such an unsorted pair $\{u_1,u_2\}$ can not exist. The same is also true in the case when only two of the three equalities hold, for instance, $i_1=i_2$ and $j_1=j_2$. Hence the only cases to consider are $i_1=i_2$, $j_1<j_2$, $k_1<k_2$ (and analogously the case $i_1<i_2$, $j_1=j_2$, $k_1<k_2$ and the case $i_1<i_2$, $j_1<j_2$, $k_1=k_2$), and the case $i_1<i_2$, $j_1<j_2$, $k_1<k_2$.
\begin{enumerate}
    \item Assume $i_1=i_2=i$, $j_1<j_2$, $k_1<k_2$. Since we are only considering unordered pairs, we can assume without loss of generality that $u_1=x_iy_*z_{k_1}$ and $u_2=x_iy_*z_{k_2}$. The only way to distribute $j$-indices between $u_1$ and $u_2$ to make $\{u_1,u_2\}$ an unsorted pair is to set $u_1=x_iy_{j_2}z_{k_1}$ and $u_2=x_iy_{j_1}z_{k_2}$. This gives the sorting relation $x_{ij_2}^{k_1}x_{ij_1}^{k_2}-x_{ij_1}^{k_1}x_{ij_2}^{k_2}$. The set $\mathcal{M}$ consists of exactly these sorting relations. We analogously obtain the sets $\mathcal{N}$ and $\mathcal{R}$.
    \item Assume $i_1<i_2$, $j_1<j_2$, $k_1<k_2$. Again, without loss of generality, we assume $u_1=x_*y_*z_{k_1}$ and $u_2=x_*y_*z_{k_2}$. There are $4$ ways to distribute the $i$-indices and $j$-indices, but only $3$ of them will produce unsorted pairs. The set $\mathcal{T}$ consists of exactly these sorting relations.
\end{enumerate}
Since there are no other ways to create unsorted pairs of monomials in $A_{mnr}$, the proof is complete.
\end{proof}

Now that we have concrete descriptions of the generating sets of both $I_{mn}^r$ and $I_{A_{mnr}}$, we are ready to show that these ideals coincide.

\begin{theorem}
  \label{thm: toric presentation}
  With all the notation as above, one has $I_{mn}^r=I_{A_{mnr}}$.
  \end{theorem}
  \begin{proof} 
   First of all note that $I_{mn}^r\subseteq I_{A_{mnr}}$. Indeed, it is not so hard to directly check that any minor belongs to $\ker \phi$, where $\phi: R\to K[A_{mnr}]$ is given by $\phi(x_{ij}^k)=x_iy_jz_k$. In order to prove the other inclusion, it is enough to show that all the minimal generators of $I_{A_{mnr}}$, whose description we obtained in \Cref{thm: mingens}, belong to $I_{mn}^r$.
   \begin{enumerate}
       \item Any binomial in $\mathcal M$ is of the form $x_{ij_1}^{k_2}x_{ij_2}^{k_1}-x_{ij_1}^{k_1}x_{ij_2}^{k_2}$ for some $j_1<j_2$, $k_1<k_2$, thus it is equal, up to a sign,  to  $\begin{vmatrix}
    x_{ij_1}^{k_1} & x_{ij_2}^{k_1} \\[1ex]
     x_{ij_1}^{k_2}   &x_{ij_2}^{k_2}
    \end{vmatrix}$, which is a minor in $V.$ 
\item Any binomial in $\mathcal N$ is of the form $x_{i_1j}^{k_2}x_{i_2j}^{k_1}-x_{i_1j}^{k_1}x_{i_2j}^{k_2}$ for some $i_1<i_2$, $k_1<k_2$, thus it is equal, up to a sign, to  $\begin{vmatrix}
    x_{i_1j}^{k_1} & x_{i_1j}^{k_2} \\[1ex]
     x_{i_2j}^{k_1}   &x_{i_2j}^{k_2}
    \end{vmatrix}$, which is a minor in $H$.
\item Any binomial in $\mathcal R$ is of the form $x_{i_1j_2}^{k}x_{i_2j_1}^{k}-x_{i_1j_1}^{k}x_{i_2j_2}^{k}$ for some $i_1<i_2$, $j_1<j_2$, thus it is equal, up to a sign, to $\begin{vmatrix}
    x_{i_1j_1}^{k} & x_{i_1j_2}^{k} \\[1ex]
     x_{i_2j_1}^{k} &x_{i_2j_2}^{k}
    \end{vmatrix}$, which is a minor in the matrix $X_k$.
    \item Any binomial in $\mathcal T$ is of one of the following three forms, where $i_1<i_2$, $j_1<j_2$, $k_1<k_2$:
    \begin{itemize}
        \item $x_{i_1j_2}^{k_1}x_{i_2j_1}^{k_2}-x_{i_1j_1}^{k_1}x_{i_2j_2}^{k_2}$: up to a sign, this is equal to $\begin{vmatrix}
    x_{i_1j_1}^{k_1} & x_{i_1j_2}^{k_1} \\[1ex]
     x_{i_2j_1}^{k_2}   &x_{i_2j_2}^{k_2}
    \end{vmatrix}$, which is a minor in $V$.
        \item $x_{i_2j_1}^{k_1}x_{i_1j_2}^{k_2}-x_{i_1j_1}^{k_1}x_{i_2j_2}^{k_2}$: up to a sign, this is equal to $\begin{vmatrix}
    x_{i_1j_1}^{k_1} & x_{i_1j_2}^{k_2} \\[1ex]
     x_{i_2j_1}^{k_1}   &x_{i_2j_2}^{k_2}
    \end{vmatrix}$,  which is a minor in $H$.
        \item $x_{i_2j_2}^{k_1}x_{i_1j_1}^{k_2}-x_{i_1j_1}^{k_1}x_{i_2j_2}^{k_2}$: up to a sign, this is equal to
 $\begin{vmatrix}
    x_{i_1j_1}^{k_1} & x_{i_1j_2}^{k_1} \\[1ex]
     x_{i_2j_1}^{k_2}   &x_{i_2j_2}^{k_2}
    \end{vmatrix}+\begin{vmatrix}
    x_{i_1j_2}^{k_1} & x_{i_1j_1}^{k_2} \\[1ex]
     x_{i_2j_2}^{k_1}   &x_{i_2j_1}^{k_2}
    \end{vmatrix}$, which is the sum of a minor in $V$ and a minor in $H$.
    \end{itemize}       
   \end{enumerate}
  \end{proof}
It is an immediate consequence that toric double determinantal rings are Hibi rings.
\begin{corollary}
\label{cor: dd are Hibi}
 $R/I_{mn}^r\cong K[L_{mnr}]$.   
\end{corollary}

\section{Invariants and Gorensteinness of toric double determinantal rings}

Since we know that toric double determinantal rings are Hibi rings, we can use this fact to compute some invariants and prove some properties about them.
Since we have to use the isomorphism $R/I_{mn}^r\cong K[L_{mnr}]$ many times, we will use  it without referring to \Cref{cor: dd are Hibi}. 

First, we would like to mention a surprising isomorphism of toric double determinantal rings which is not at all obvious from the definition of $I_{mn}^r$.

\begin{theorem}
\label{thm: isomorphism}
Let $\sigma$ be any permutation of $m,n,r$.
Then $R/I_{mn}^r\cong R/I_{\sigma(m)\sigma(n)}^
{\sigma(r)}$.
\end{theorem}
\begin{proof}
  Clearly, $P_{mnr}$ is isomorphic to $P_{\sigma(m)\sigma(n)\sigma(r)}$, thus so are their Hibi rings. Hence we get $R/I_{mn}^r\cong K[L_{mnr}]\cong K[L_{\sigma(m)\sigma(n)\sigma(r)}]\cong R/I_{\sigma(m)\sigma(n)}^
{\sigma(r)}$. 
\end{proof}
\begin{remark}
 \begin{enumerate}
\item  It is possible to get the above result by proving that the map from $R$ to $R$  sending $x_{ij}^k$ to $x_{\sigma(i)\sigma(j)}^
{\sigma(k)}$ induces an isomorphism between $I_{mn}^r$ and $I_{\sigma(m)\sigma(n)}^
{\sigma(r)}$.  However, the situation gets complicated due to the fact that not all minors of $I_{mn}^r$ map to minors (but rather linear combinations of minors) of $I_{\sigma(m)\sigma(n)}^
{\sigma(r)}$.

\item The isomorphism does not hold for double determinantal ideals generated by larger minors. Indeed, using Macaulay2 \cite{M2}, one can check that if we take $t_h=t_v=3$, then $\HS(R/I_{33}^4)\neq\HS(R/I_{34}^3)$.
 Even in the case of $2$-minors, the isomorphism in Theorem \ref{thm: isomorphism} does not hold for the initial ideals with respect to a diagonal order (see Section 4 for the definition).  
 In fact, one can check using Macaulay2 \cite{M2} that $\ini(I_{22}^3)$ and $\ini(I_{23}^2)$ have different resolutions, hence they are not isomorphic.
 \item On the contrary, if we consider any sorting order (which is the same as any Hibi order) $\tau$, we will obtain $R/\ini_{\tau}(I_{mn}^r)\simeq R/\ini_{\tau}(I_{\sigma(m)\sigma(n)}^{\sigma(r)})$. Indeed, sorting relations are mapped into sorting relations (in which unsorted pairs are mapped to unsorted pairs) under any permutation of $\{m,n,r\}$.  This is because the sorting process takes care of the $x$'s, $y$'s, and $z$'s independently. 
 \end{enumerate}
\end{remark}

In the next theorem we determine some invariants of toric double determinantal rings. Note that the dimension of $R/I_{mn}^r$ has been calculated before for any double determinantal ring, see \Cref{prop: CM}.

\begin{theorem} 
\label{thm: invariants}
The following equalities hold: 
\begin{enumerate}

\smallskip
\item $\mu(I_{mn}^r)
={mnr+1\choose 2}-{m+1\choose 2}{n+1\choose 2} {r+1\choose 2}$, where $\mu(\cdot)$ denotes the number of minimal generators.
\smallskip
\item $\dim R/I_{mn}^r=m+n+r-2$.
\smallskip
\item  $\e(R/I_{mn}^r)=\frac{(m+n+r-3)!}{(m-1)!(n-1)!(r-1)!}$, where $\e(\cdot)$ denotes the multiplicity. 
\smallskip
\item   $\reg(R/I_{mn}^r)=m+n+r-2-\max\{m,n,r\}$.
\smallskip
\item $\ainv(R/I_{mn}^r)=-\max\{m,n,r\}$, where $\ainv(\cdot)$ denotes the $a$-invariant.
    \medskip
    \item $\HF(R/I_{mn}^r, d)=\binom{m-1+d}{d}\binom{n-1+d}{t}\binom{r-1+d}{d}$ for all $d\geq 0$, where $\HF(\cdot, \cdot)$ denotes the Hilbert function.
\end{enumerate}
\end{theorem}
\begin{proof}
\begin{enumerate}        
        \item We know that $I_{mn}^r$ is isomorphic to the presentation ideal of $K[L_{mnr}]$, which in turn is minimally generated by Hibi relations (see \Cref{rem: hibiorder}(2)), indexed by incomparable pairs of $L_{mnr}$.
        Therefore, in order to compute $\mu(I_{mn}^r)$, it is enough to compute the number of incomparable pairs of $L_{mnr}$ (see \Cref{fig: parallelepiped} for a better overview of the situation). Since $L_{mnr}$ has $mnr$ elements in total, there are ${mnr+1} \choose 2$ ways to choose a pair of (not necessarily distinct) elements. We will now compute in how many ways we can choose a pair of comparable elements, and subtract these two. For each element labeled $x_iy_jz_k$, all the elements smaller or equal to this element are located in the parallelepiped between $x_1y_1z_1$ and $x_iy_jz_k$, thus there are $ijk$ elements less than or equal to our chosen element. Hence, the number of pairs of (not necessarily distinct) comparable elements is 
        $$
             \sum_{i=1}^{m} \sum_{j=1}^{n} \sum_{k=1}^{r}{ijk}=\sum_{i=1}^{m}{i}
            \sum_{j=1}^{n}{j}\sum_{k=1}^{r}{k}=
            {{m+1} \choose 2} {{n+1} \choose 2} {{r+1} \choose 2}.
        $$
        Subtracting the two numbers gives us $\mu(I_{mn}^r)={mnr+1\choose 2}-{m+1\choose 2}{n+1\choose 2} {r+1\choose 2}$. 
           \smallskip
        \item Since $\dim K[\J(P)]=|P|+1$ (see \cite{Hi}), we get $\dim R/I_{mn}^r=\dim K[L_{mnr}]=\dim K[\J(P_{mnr})]=|P_{mnr}|+1=m+n+r-2$.
           \smallskip
        \item Let $L=\J(P)$. From \cite[Proposition~2.3]{HiHi} it follows the multiplicity of $K[L]$ equals the total number of linear extensions of $P$, which are in bijection with maximal chains of $L$. In our case these are all the paths from $x_1y_1z_1$ to $x_{m}y_{n}z_{r}$ directed upwards. Each such path can be encoded with a word on the multiset $\{M^{m-1}, N^{n-1},R^{r-1}\}$: every edge of $L_{mnr}$ changes either the $x$-index by $1$ (in which case we label it with $M$), or similarly the $y$-index (in which case we label it with $N$), or the $z$-index (in which case we label it with $R$). Then every upwards directed path from $x_1y_1z_1$ to $x_{m}y_{n}z_{r}$ can be encoded with the concatenation of the labels of its edges.
        Clearly, this gives a bijection between all the maximal chains of $L_{mnr}$ and the set of the words on the multiset $\{M^{m-1}, N^{n-1},R^{r-1}\}$,  and there are exactly $\frac{(m+n+r-3)!}{(m-1)!(n-1)!(r-1)!}$ of them. Indeed, the words are on $m+n+r-3$ symbols, hence there are $(m+n+r-3)!$ permutations; the result follows because any permutation on all the $M$'s gives the same word, and the same holds for $N$'s and $R$'s. 
        \smallskip
        \item The regularity of $K[\J(P)]$ is known to be $|P|-\rank(P)-1$ (see \cite[Theorem~1.1]{EHM}). Recall that $\rank(P)$ is the cardinality of the longest maximal chain of $P$ minus $1$. In our case $|P_{mnr}|=m+n+r-3$ and $\rank(P_{mnr})=\max\{m-1,n-1,r-1\}-1=\max\{m,n,r\}-2$. Hence $\reg (R/I_{mn}^r)=m+n+r-3-(\max\{m, n,r\} -2)-1=m+n+r-2-\max\{m,n,r\}$. 
           \smallskip
        \item Since $R/I^r_{mn}$ is Cohen-Macaulay, \[\ainv(R/I_{mn}^r)=\reg(R/I_{mn}^r)-\dim R/I_{mn}^r=-\max\{m,n,r\}.\] 
           \smallskip
        \item It is known that $\HF(K[J(P)],d)$ is the number of order-preserving maps $P\to \{0,\ldots, d\}$, see \cite[Proposition 2.3]{HiHi}, combined with \cite[Theorem 3.15.8]{Stn} (note that Stanley uses a slightly different definition of an order-preserving map, hence our formulas look different from his). Since the chains $A_1$, $A_2$, and $A_3$ are fully incomparable to each other, we can treat them separately. Any order-preserving map from $A_1$ to $\{0,\ldots, d\}$ can be encoded with the vector 
        $(a_1,a_2,\ldots,a_{m-1})$, where $0\le a_1\le a_2\le\ldots\le a_{m-1}\le d$. This vector can be uniquely recovered from the vector $(a_1,a_2-a_1,a_3-a_2,\ldots,a_{m-1}-a_{m-2},d-a_{m-1})\in\mathbb{N}^m$, or, equivalently, from the monomial $x_1^{a_1}x_2^{a_2-a_1}\cdots x_m^{d-a_{m-1}}\in K[x_1,\ldots,x_m]_d$. It is not so hard to see that if we argue in the same way for the chains $A_2$ and $A_3$, we get a bijection between the set of all the order-preserving maps $P_{mnr}\to \{0,\ldots, d\}$ and the set of  triples of monomials in $K[x_1,\ldots,x_m]_d\times K[y_1,\ldots,y_n]_d \times K[z_1,\ldots,z_r]_d$, and there are precisely $\binom{m-1+d}{d}\binom{n-1+d}{d}\binom{r-1+d}{d}$ of them.
\end{enumerate}
\end{proof}
\begin{remark}
\label{rmk:HSnotequal}
\begin{enumerate}
\item Note that by \Cref{thm: mingens}, the ideal $I_{A_{mnr}}$, which we know equals $I_{mn}^r$, is minimally generated by $\mathcal M\cup\mathcal N\cup\mathcal R\cup\mathcal T$. It is clear that these are disjoint sets of cardinality $m {r\choose 2} {n\choose 2}$, $n{r \choose 2}{m\choose 2}$, 
        $r{m \choose 2} {n\choose 2}$ and $3{r\choose 2}{m \choose 2}{n\choose 2}$, respectively. This gives a different way to calculate $\mu(I_{mn}^r)$, that is 
        $$\mu(I_{mn}^r)= m {r\choose 2} {n\choose 2}+n{r \choose 2}{m\choose 2}+r{m \choose 2} {n\choose 2}+3{r\choose 2}{m \choose 2}{n\choose 2}.$$
        One can easily check that this equals the formula  in \Cref{thm: invariants}(1).
        %${mnr+1\choose 2}-{m+1\choose 2}{n+1\choose 2} {r+1\choose 2}$.
    \item  Another way to prove \Cref{thm: invariants}(6) is to use $R/I_{mn}^r\cong K[A_{mnr}]$, see \Cref{thm: toric presentation}. Recall that $A_{mnr}$ is generated by all the monomials of the form $u_{ijk}=x_iy_jz_k$, with $ 1\le i\le m$, $ 1\le j\le n$, $1\le k\le r$, and that $\deg u_{ijk}=1$ in $K[A_{mnr}]$ for all $i,j,k$.
Thus a basis of $K[A_{mnr}]_d$ is given by all the products $u_{i_1j_1k_1}u_{i_2j_2k_2}\cdots u_{i_dj_dk_d}=(x_{i_1}\cdots x_{i_d})(y_{j_1}\cdots y_{j_d})(z_{k_1}\cdots z_{k_d})$. 
The set of these monomials is in an obvious bijection with the set of triples of monomials in $K[x_1,\ldots,x_m]_d\times K[y_1,\ldots,y_n]_d \times K[z_1,\ldots,z_r]_d$, hence there are precisely $\binom{m-1+d}{d}\binom{n-1+d}{d}\binom{r-1+d}{d}$ of them.
    \item All the formulas in \Cref{thm: invariants} are symmetric in $m$, $n$, $r$, which is consistent with \Cref{thm: isomorphism}.

 \end{enumerate}
\end{remark}

In order to give a formula for the $h$-polynomial (which is the numerator of the Hilbert series in its reduced form) of $R/I_{mn}^r$, let us index the elements of $P=\{p_1, \ldots, p_{|P|}\}$ so that $p_i\prec p_j$ implies $i<j$. In other words, we are giving $P$ a natural labeling.  Let $\mathcal{L}(P)$ be the set of the linear extensions of $P$, meaning the set of bijections $w: P \to \{1,2, \ldots, |P|\}$  respecting the partial order on $P$. A \emph{descent} of $w$ is an index $i$ such that $w(p_j)=i$ and $w(p_k)=i+1$ where $j>k$, and $\operatorname{des}(w)$ denotes the number of descents of $w$. 
It is pointed out in \cite[Proposition~2.3]{HiHi} that if $L=\J(P)$, then the Hilbert series of $K[L]$ is

\begin{equation}
\label{eq:Hilbert_series}
\HS(K[L],t)=\frac{\sum_{w \in \mathcal{L}(P)}t^{\operatorname{des}(\omega)}}{(1-t)^{|P|+1}}.
\end{equation}
By using this fact, one obtains a combinatorial formula for the $h$-polynomial of the toric double determinantal rings.
\begin{corollary}
    The $h$-polynomial of $R/I_{mn}^r$ is equal to $\sum_{w \in \mathcal{L}(P_{mnr})}t^{\operatorname{des}(\omega)}$.
\end{corollary}
\begin{proof}
    It follows immediately from \Cref{eq:Hilbert_series}, since $\dim(R/I_{mn}^r)=|P_{mnr}|+1$.
\end{proof}
We will now give another interpretation of the $h$-vector that is easily seen to be equivalent to the one above.
For nonnegative integers $a_1,a_2,\ldots,a_s$, let $M=\{1^{a_1},2^{a_2},\ldots, s^{a_s}\}$ and let $a=a_1+\cdots+a_s$. Let $S(M)$ denote the set of all the permutations on this multiset. In other words, it is a set of strings that contains exactly $a_1$ ones, $a_2$ twos, and so on. 
Given such a string $\sigma=\sigma_1\cdots \sigma_a$, a descent of $\sigma$ is an index $k$ such that $\sigma_{k}>\sigma_{k+1}$.  We denote by $\des(\sigma)$ the number of descents of $\sigma$. 

The next proposition is due to MacMahon \cite[Volume 2, p. 211]{MM}, see also \cite[Equation 1.1]{LXZ}.
\begin{proposition}
\label{McMahon}
$\sum_{\sigma\in S(M)}{t^{\des(\sigma)}}=(1-t)^{a+1}\sum_{d\ge 0}\binom{a_1+d}{d}\cdots\binom{a_s+d}{d}t^d$.
\end{proposition}
As a consequence, we get another description for the $h$-polynomial of $R/I_{mn}^r$.
\begin{corollary}
  Let $M=\{1^{m-1},2^{n-1},3^{r-1}\}$. Then the $h$-polynomial of $R/I_{mn}^r$ is $\sum_{\sigma\in S(M)}{t^{\des(\sigma)}}$.
\end{corollary}
\begin{proof}
If we set $s=3$, $a_1=m-1$, $a_2=n-1$, $a_3=r-1$ in \Cref{McMahon}, we obtain 
$$
\frac{\sum_{\sigma\in S(M)}{t^{\des(\sigma)}}}{(1-t)^{m+n+r-2}}=\sum_{d\ge 0}\binom{m-1+d}{d}\binom{n-1+d}{d}\binom{r-1+d}{d}t^d,
$$
whose right hand side is the Hilbert series of $R/I_{mn}^r$ as we showed in \Cref{thm: invariants}. Since $\dim R/I_{mn}^r=m+n+r-2$ by \Cref{thm: invariants}(2), the proof is complete. 
\end{proof}
\begin{remark} 
It is not hard to see why the two combinatorial formulas for the $h$-polynomial of $R/I_{mn}^r$ coincide. One of the sums runs over all the linear extensions of $P_{mnr}$, the other one runs over all the elements in the multiset $\{1^{m-1},2^{n-1},3^{r-1}\}$. In order to show that the two polynomials coincide, it is sufficient to find a bijection between the two sets that preserves the number of descents. First of all we give our $P_{mnr}$ the following natural labeling: elements of $\red{A_1}$ are labeled $\red{p_1},\ldots, \red{p_{m-1}}$, elements of $\green{A_2}$ are labeled $\green{p_m},\green{p_{m+1}}\ldots, \green{p_{m+n-2}}$, and finally elements of $\blue{A_3}$ are labeled $\blue{p_{m+n-1}},\ldots, \blue{p_{m+n+r-3}}$. %It is known that linear extensions of any poset $P$ are in bijection with maximal chains of $\J(P)$. 

We will show how this bijection works in the case  $m=3$, $n=2$, $r=4$ (see \Cref{ex: hibisortable} and \Cref{fig: parallelepiped}), but the same idea works for any $P_{mnr}$. With the new labeling  introduced above, the elements of $P_{\red{3}\green{2}\blue{4}}$ are labeled $\red{p_1},\red{p_2},\green{p_3},\blue{p_4}, \blue{p_5},\blue{p_6}$ (see \Cref{fig: labelling}).
\begin{figure}[H]
\centering
\begin{tikzpicture}[scale=0.8]

\draw[thick,red] (-12,0)--(-12,2);
\filldraw[red] (-12,0) circle (2pt);
\filldraw[red] (-12,2) circle (2pt);
\node [thin, left] at (-12,0) {$\red{p_1}$};
\node [thin, left] at (-12,2) {$\red{p_2}$};
\node [thin, red] at (-12,-1) {$A_1$};

\filldraw[green] (-10,0) circle (2pt);
\node [thin, left] at (-10,0) {$\green{p_3}$};
\node [thin, green] at (-10,-1) {$A_2$};

\draw[thick,blue] (-8,0)--(-8,4);
\filldraw[blue] (-8,0) circle (2pt);
\filldraw[blue] (-8,2) circle (2pt);
\filldraw[blue] (-8,4) circle (2pt);
\node [thin, left] at (-8,0) {$\blue{p_4}$};
\node [thin, left] at (-8,2) {$\blue{p_5}$};
\node [thin, left] at (-8,4) {$\blue{p_6}$};
\node [thin, blue] at (-8,-1) {$A_3$};
\end{tikzpicture}
\caption{A natural labeling of $P_{\red{3}\green{2}\blue{4}}$}
\label{fig: labelling}
\end{figure}
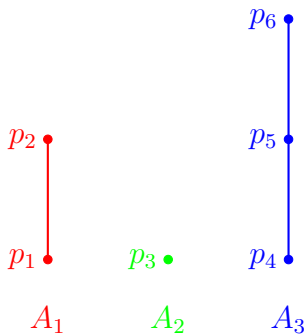
Consider the linear extension $\red{p_1}\preceq \green{p_3}\preceq\blue{p_4}\preceq\red{p_2}\preceq\blue{p_5}\preceq\blue{p_6}$. In other words, $w(\red{p_1})=1$, $w(\green{p_3})=2$ etc. The descent set of this extension can be found as the descent set of the permutation $\red{1}\green{3}\blue{4}\red{2}\blue{5}\blue{6}$. We will abuse notation and say $w=\red{1}\green{3}\blue{4}\red{2}\blue{5}\blue{6}$ for convenience.
To this permutation $w$ we associate the word $\sigma=123133\in \{1^2,2^1,3^3\}$ simply by replacing  the red numbers with $1$'s, the green numbers with $2$'s, and the blue numbers with $3$'s. We can uniquely recover $w$ from any $\sigma$. Indeed, first of all, $\sigma$ tells us the colors, in a sense that the permutation we are looking for has to be of the form $\red{*}\green{*}\blue{*}\red{*}\blue{*}\blue{*}$. Since the result should be a linear extension, meaning that it has to respect the partial order of $P_{mnr}$, the red labels all have to go in the increasing order (since $\red{p_1}\preceq \red{p_2}$ in $P_{mnr}$, the same must hold in any linear extension). The same holds for the green and blue labels, hence there is exactly one way to get $w$ for any $\sigma$. It is not hard to notice that $w$ (in this case $\red{1}\green{3}\blue{4}\red{2}\blue{5}\blue{6}$) and $\sigma$ (in this case $123133$) have the same descent set. Indeed, all the descents in $\sigma$ are of the shapes $21$, $32$, $31$. The descent $21$ translates to a subword of sorts $\green{*}\red{*}$, and any green label is bigger than any red one, hence it is a descent in $w$. The same holds for descents of sorts $32$ and $31$. Conversely, consider a descent in $w$. It can not happen between numbers of the same color since $w$ is a linear extension (as discussed above). Since any blue label is bigger than any green label, and any green label is bigger than any red label, the only possible descents in $w$ are of the forms $\green{*}\red{*}$ (which translates to a descent $21$ in $\sigma$), $\blue{*}\green{*}$ (32) and $\blue{*}\red{*}$ (31). 

\end{remark}

\smallskip

Finally, we address the question of the Gorenstein property of $R/I_{mn}^r$.
\begin{theorem}
\label{thm: Gorenstein}
 The ring $R/I_{mn}^r$ is Gorenstein if and only if $\{m,n,r\}\subseteq \{1,\max\{m,n,r\}\}$.
 \end{theorem}
 \begin{proof}
  From \cite{Hi} we know that $K[L_{mnr}]$ is Gorenstein if and only if the corresponding poset $P_{mnr}$ is pure. In other words, all maximal chains of $P_{mnr}$ have to be of the same length or, equivalently, of the same cardinality. Clearly, the cardinality of the longest chain in $P_{mnr}$ is $\max\{m-1,n-1,r-1\}$. The cardinality of $A_1$ is $m-1$. This chain either needs to have the maximal cardinality, that is $m-1=\max\{m-1,n-1,r-1\}$, or be empty, that is, $m-1=0$. That is, either $m=\max\{m,n,r\}$ or $m=1$. The same holds for the chains $A_2$ and $A_3$, and this concludes the proof.
 \end{proof}

%%%%%%%%%%%%%%%% 
\section{The simplicial complex associated to double determinantal ideals}
The following result was proved by  Fieldsteel and Klein \cite[Theorem 3.2]{FK}, for a double determinantal ideal generated by minors of any size:
\begin{theorem}
\label{thm:GB}
The minors generating $I_{mn}^r$ are a Gr\"obner basis with respect to any diagonal order.
\end{theorem} 

In this context, we call a term order diagonal if the leading term of any minor of $H$ and $V$ is the product of the indeterminates of the main diagonal. For example the lexicographic order with $x_{11}^{1}>x_{12}^{1}>\cdots>x_{1n}^{1}>x_{21}^{1}>\cdots>x_{2n}^{1}>\cdots >x_{mn}^{1}>x_{11}^{2}>\cdots >x_{mn}^{2}>\cdots >x_{mn}^{r}$ is a diagonal term order.  For the remainder of this paper, all initial ideals will be with respect to a diagonal term order.

Throughout this section, we will use the same notation as in the previous section, and consider the ideals $I_{mn}^r$ in the case $t_h=t_v=2$.  Moreover we will denote by $\Delta$ the simplicial complex associated to $\ini(I_{mn}^r)$, that is, the simplicial complex such that $\ini(I_{mn}^r)=I_{\Delta}$, that is $R/\ini(I_{mn}^r)=K[\Delta]$ is the Stanley-Reisner ring associated to $\Delta$ (see \cite[Section 1.5]{HH}).  In the remainder of this section we give two descriptions of the facets of $\Delta$.

\smallskip
By \Cref{thm:GB} the ideal $\ini(I_{mn}^r)$ is generated by the $2$-diagonals in the horizontal and in the vertical concatenation of $X_1,\ldots,X_r$, with $X_k=(x_{ij}^k)$, that is, by the monomials $x_{ij}^kx_{i'j'}^{k'}$ such that one of the following conditions holds
\begin{itemize}
    \item[(a)] $k=k'$, $i'>i$ and $j'>j$, 
    \item[(b)] $k<k'$, $i<i'$, any $j$, $j'$,  
    \item[(c)] $k<k'$, $j<j'$, any $i$, $i'$.
\end{itemize}
Note that in the above conditions, (a)
describes the diagonal of a $2$-minor which occurs in one of $X_1, \dots, X_r$, (b) describes the diagonal of a $2$-minor whose columns occur in two different matrices in the horizontal concatenation, and (c) describes the diagonal of a $2$-minor which whose rows appear in two different matrices in the vertical concatenation.

To describe $\Delta$ we identify the variables of $R = K[x_{ij}^k\,|\, 1\le i\le m,\ 1\le j\le n, \ 1 \le k \le r]$ with the set
$$\mathcal{P}=\{(i,j):1\leq i\leq m, \ 1\leq j\leq nr\}$$
by associating  $x_{ij}^k$ to $(i,(k-1)n+j$ for every $i,j$ and $k$, and we use $\mathcal{P}$ as the vertex set of $\Delta.$ Thus the faces of $\Delta$ are the subsets of $\mathcal{P}$ that don't contain any pair of points $P_1=(i,(k-1)n+j),\ P_2=(i',(k'-1)n+j')$ such that one of the conditions (a), (b), and (c) above hold.

Denote by $\mathcal{F}(\Delta)$ the set of the facets of $\Delta$.  In the theorem below, we give a description of the facets as the union of paths. 

\begin{definition}
    A subset $\{(g_1,h_1), \ldots, (g_t, h_t) \}$ of elements of $\mathcal{P}$ is said to be a {\it path} with starting point $(g_1,h_1)$ and ending point $(g_t, h_t)$ if $(g_{s+1}, h_{s+1}) - (g_s, h_s) \in \{(-1,0), (0,1)\}$ for all $s=1,\ldots, t-1$. We denote it by $(g_1,h_1)\longrightarrow (g_s, h_s)$. 
\end{definition} 

\begin{remark}\label{completingPaths}
Note that any pair of points $(g, h)$, $(g', h')$ with the property that $g \geq g'$ and $h \leq h'$ lie on a path.  Specifically, they lie on the path $\{(g, h), (g, h+1), (g, h+2), \dots, (g, h'), (g -1, h'), (g-2, h'), \dots, (g', h')$. Thus by extension, there is a path which contains any sequence of points $(g_1, h_1), (g_2, h_2), \dots, (g_t, h_t)$ provided that $g_1 \geq g_2 \geq \dots \geq g_t$ and $h_1 \leq h_2 \leq \dots \leq h_t$.
\end{remark}
With an abuse of notation, in the following we consider also a singleton $(g,h)$ as a path $(g,h)\longrightarrow (g,h)$. 
  \begin{theorem}\label{facets}
      Let $\ini(I_{mn}^r)=I_\Delta$, with $\Delta$ simplicial complex on $\mathcal{P}$. Then $F\in \mathcal{F}(\Delta)$ if and only if $F$ is the union of the following $r$ paths:      
 \[
  	(g_{k-1}, (k-1)n+h_{k}) \longrightarrow (g_{k},(k-1)n+h_{k-1}) \mbox{ for $1\leq k\leq r$,}
\]  
    for some $m=g_0\geq g_1\geq g_2\geq \cdots \geq g_{r-1}\geq g_r= 1$ and $n=h_0\geq h_1\geq h_2\geq\cdots \geq h_{r-1}\geq h_r= 1$.
  \end{theorem}
 
  \begin{proof}
  Let $F$ be a union of paths as described above.  To see that $F$ is a face of $\Delta$, we must show that for any two points  $P_1$ and $P_2$ of $F$, $P_1$ and $P_2$ do not satisfy conditions (a), (b), and (c) above.  Let $P_1 = (i_1, (k_1-1)n+j_1)$ and $P_2 = (i_2, (k_2-1)n+j_2)$ and assume $k_1 \leq k_2$.   If $k_1=k_2$, then $P_1$ and $P_2$ are on the same path, and by definition, we have $i_1 >i_2$ and $j_1 <j_2$ (or vice versa). Thus, $P_1$ and $P_2$ do not satisfy condition (a).  If $k_1 < k_2$ then $P_1$ is on the path $(g_{k_1-1}, (k_1-1)n+h_{k_1}) \longrightarrow (g_{k_1}, (k_1-1)n+h_{k_1-1})$ and so $g_{k_1-1} \geq i_1 \geq g_{k_1}$ and $ h_{k_1} \leq j_1 \leq h_{k_1-1}$.  Similarly $P_2$ is on the path  $(g_{k_2-1}, (k_2-1)n+h_{k_2}) \to (g_{k_2}, (k_2-1)n+h_{k_2-1})$ and so $g_{k_2-1} \geq i_2 \geq g_{k_2}$ and $ h_{k_2} \leq j_2 \leq h_{k_2-1}$.  By assumption $g_{k_1} \geq g_{k_2-1}$ and $h_{k_1} \geq h_{k_2-1}$, which implies that $i_1 \geq i_2$ and $j_1 \geq j_2$. Thus, $P_1$ and $P_2$ do not satisfy conditions (b) and (c). Therefore $F \in \Delta$. 
     
      Since the cardinality of $F$ is
      $\sum_{k=1}^{r}((g_{k-1}-g_k)+(h_{k-1}-h_k)+1)= g_0- g_r+h_0-h_r+r = m+n+r-2=\dim R/\ini(I_{mn}^r)$, by \cite[Corollary 6.2.2]{HH} it follows that $F$ is a facet. 
      
It remains to prove that no face of $\Delta$ that is not of this form is a facet, that is, that every $E\in\Delta$ is contained in a union of paths as described above. Let $E=\{P_1, \dots, P_t\} \in \Delta$ where for each $1\leq s\leq t$, $P_s = (i_s, (k_s-1)n+j_s)$ for some $1\leq i_s \leq m$, $ 1 \leq j_s \leq n$, and $1\leq k_s \leq r$.  We may assume $k_1 \leq \dots \leq k_t$. Since $E$ does not contain any diagonal elements, by the above conditions, for any pair of elements $P_s$ and $P_{s'}$ with $s<s'$ we have either $k_s <k_{s'}$, in which case $i_s \geq i_{s'}$ and $j_s \geq j_{s'}$, or $k_s = k_{s'}$, in which case either $i_s \leq i_{s'}$ and $j_s \geq j_{s'}$ or $i_s  \geq i_{s'}$ and $j_s \leq j_{s'}$.  Without loss of generality (reordering if necessary), we may assume that when $k_s = k_{s'}$, the latter is true.  Thus we have $k_1 \leq k_2 \leq \dots \leq k_t$, $i_1\geq i_2 \geq  \dots \geq  i_t$ and for any $s<s'$, $j_s \leq j_{s'}$ if $k_s = k_{s'}$ and $j_s \geq j_{s'}$ if $k_s < k_{s'}$.  In addition, since any facet has cardinality at most $m+n+r-2$, we have $t\leq m+n+r-2$. 

We form a partition of $E$ based on the values of $k_1 \dots, k_t$.  That is, for each $1 \leq \ell \leq t$ define $E_\ell = \{(i, (k-1)n+j) \in E \mid k = \ell\}$ (note that $E_\ell$ may be empty). Further, for each $1 \leq \ell \leq t$ for which $E_\ell$ is nonempty, define $\alpha_\ell$ to be the smallest index such that $P_{\alpha_\ell} \in E_\ell$ and define $\beta_\ell$ to be the largest index such that $P_{\beta_\ell} \in E_\ell$.  Thus $E_\ell = \{P_{\alpha_\ell}, P_{\alpha_\ell+1}, \dots, P_{\beta_\ell}\}$.  Note that $\alpha_1 = 1$ and $\beta_r = t$ and that for any $\ell$, $i_{\alpha_\ell} \geq i_{\alpha_\ell+1} \geq \dots \geq i_{\beta_\ell}$ and $j_{\alpha_\ell} \leq j_{\alpha_\ell +1}\leq \dots \leq j_{\beta_\ell}$.

Suppose all of the sets $E_1, \dots, E_r$ are nonempty.  Then we can extend $E$ to a facet of $\Delta$ as follows:

Since $m\geq i_1= i_{\alpha_1} \geq i_{\alpha_1+1} \geq \dots \geq i_{\beta_1}\geq i_{\alpha_2}$ and $j_{\alpha_1}\leq j_{\alpha_1+1}\leq \dots \leq j_{\beta_1}\leq n$, by \Cref{completingPaths} there is path from $(m, j_{\alpha_1})$ to $(i_{\alpha_2}, n)$ which contains all of the elements of $E_1$.  
  
 For each $2 \leq \ell \leq r-1$, since $i_{\alpha_\ell} \geq i_{\alpha_\ell+1} \geq \dots \geq i_{\beta_\ell} \geq i_{\alpha_{\ell+1}}$ and $j_{\alpha_\ell} \leq j_{\alpha_\ell+1} \leq \dots \leq j_{\beta_\ell} \leq j_{\alpha_{\ell-1}}$, by \Cref{completingPaths}, there is a path from $(i_{\alpha_\ell}, (\ell-1)n+j_{\alpha_\ell})$ to $(i_{\alpha_{\ell+1}}, (\ell-1)n+j_{\alpha_{\ell-1}})$ which contains the elements of $E_\ell$.

Finally, since $i_{\alpha_r} \geq i_{\alpha_r+1} \geq \dots \geq i_{\beta_r} \geq  1$ and $1\leq j_{\alpha_{r}} \leq j_{\alpha_{r}+1} \leq \dots \leq j_{\beta_r} \leq j_{\alpha_{r-1}}$, by \Cref{completingPaths} there is a path from $(i_{\alpha_r}, 1)$ and $(1, (r-1)n+j_{\alpha_{r-1}})$ which contains the elements of $E_r$.  

This union of paths is a facet of the desired form which contains $E$.

Now suppose that some of the $E_1, \dots, E_r$ are empty.  We form the paths as described above for all non-empty sets $E_\ell$, skipping over any empty sets.  Then, for each empty $E_\ell$, we add a single point to $E_\ell$ as follows.  Let $a$ be the smallest index such that $E_a \neq \emptyset$ and let $b$ be the largest such index.  Suppose $E_\ell = \emptyset$.  If $\ell <a$ then we add the point $(m,n)$ to $E_\ell$.  If $\ell >b$ we add $(1,1)$ to $E_\ell$.  If $a < \ell <b$ then let $c$ be the largest index such that $c<\ell$ and $E_c \neq \emptyset$ and let $d$ be the smallest index such that $\ell <d$ and $E_d \neq \emptyset$.  Then we add the point $(i_{\alpha_d}, j_{\alpha_c})$ to $E_\ell$.  Again, the union of the paths and the isolated points is a facet of the desired form.  
\end{proof}

\begin{remark}
\begin{enumerate}
   \item   Every $F\in\mathcal F(\Delta)$ is the union of $r$ paths, and each one in the subregion of $\mathcal{P}$ corresponding to one of the  matrices $X_1,\ldots, X_r$.
   \item The complex  $\Delta$ is pure, since $|F|=m+n+r-2$ for every $F\in\mathcal{F}(\Delta)$. This has been proved in a more general case in \cite{L}, where $\Delta$ is also proved to be shellable, but no explicit formulas for the Hilbert series and for the multiplicity are given.
\end{enumerate}
\end{remark}

\begin{example}
\label{ex:complete-faces}
Let $m=4$, $n=5$, and $r=3$, and let 
$$E=\{(4,5),(3,7),(2,8),(2,11),(1,12)\}\in\Delta.$$
The points in $E$ are the ones with red circles in \Cref{fig: facet extension}, where the points are represented  not only in $\mathcal{P}$ but also in the vertical region. This is done so that the reader can check that the facet we will eventually obtain creates a diagonal neither in the horizontal, nor in the vertical concatenation, and that no further points can be added to the facet. Following the proof of \Cref{facets}, we extend $E$ to a facet $F\in\mathcal{F}(\Delta)$. Since we have some elements of $E$ in each matrix, we start adding elements to the first matrix. The element $(4,5)$ matches with the initial element of the first path described as $(m,h_1)$ above, hence we set $(4,5)$ as the initial element of the first path. From here we know $h_1=5$. Next, we set the initial element of the second path as $(3,7)$ given in the second matrix. Then we know $g_1=3$ and $h_2=2$. Hence, we add the element $(3,5)$ as the last element of the first path that matches with $(g_1,n)$. Next, we set the initial element of the last matrix as $(2,11)$. So we know that $g_2=2$ and $h_3=1$. Hence, we add the element $(2,10)$ as the last element of the second path that matches with $(g_2,(2-1)5+h_1)$. By adding elements $(2,7), (2,9)$, we complete the second path. Finally, since we know the last element of the last path should match with $(g_3,(3-1)5+h_2)$, we set the last element as $(1,12)$ given in $E$. Thus, $g_3=1$. By adding the element $(1,11)$, we complete the last path. By the \Cref{prop: CM}, we know the cardinality of any facet is $m+n+r-2=10$ and since $|F|=10$, it follows that $F$ is a facet. This can be seen also by noting that $F$ is the union of three paths: $(4,5)\longrightarrow (3,5)$, $(3,7)\longrightarrow (2,10)$, and $(2,11)\longrightarrow (1,12)$.
\begin{figure}[h]
    \begin{tikzpicture}[scale=0.9]
     \foreach \x in {1,...,15}
     \foreach \y in {1,...,4}
     {
     \filldraw[black] (\x,\y) circle (2pt);
     }
 \foreach \x in {1,...,5}
     \foreach \y in {-7,...,1}
     {
     \filldraw[black] (\x,\y) circle (2pt);
     }
\draw (5.5,4.5)--(5.5,-7.5);
\draw (10.5,4.5)--(10.5,0.5);
\draw (0.5,0.5)--(15.5,0.5);
\draw (0.5,-3.5)--(5.5,-3.5);
        \draw[red] (5,1) circle (6pt);
        \draw[blue] (5,2)circle (6pt);
        \draw[red] (7,2) circle (6pt);
        \draw[blue] (7,3) circle (6pt);        
        \draw[red] (8,3) circle (6pt);
         \draw[blue] (9,3) circle (6pt);
        \draw[blue] (10,3) circle (6pt);
        %\draw[red] (8,3) circle (9pt);
        \draw[red] (11,3) circle (6pt);
         \draw[blue] (11,4) circle (6pt);
         \draw[red] (12,4) circle (6pt);
         \draw[red] (3,-1) circle (6pt);
          \draw[blue] (2,-1) circle (6pt);
           \draw[blue] (4,-1) circle (6pt);
            \draw[blue] (5,-1) circle (6pt);
         %\draw[red] (3,-1) circle (9pt);
         \draw[red] (2,-2) circle (6pt);
          \draw[blue] (1,-4) circle (6pt);
         \draw[red] (2,-4) circle (6pt);
         \draw[red] (1,-5) circle (6pt);
         %\node[thin,above] at (5,4.5) {$x$};
         \node[thin,above] at (12,4.5) {$_{h_2=2}$};
          \node[thin,above] at (11,4.5) {$_{h_3=1}$};
          \node[thin,above] at (10,4.5) {$_{h_1=5}$};
          \node[thin,above] at (7,4.5) {$_{h_2=2}$};
          \node[thin,above] at (5,4.5) {$_{h_1=h_0=5}$};
          \node[thin,left] at (0,1) {$_{g_0=4}$};
           \node[thin,left] at (0,2) {$_{g_1=3}$};
           \node[thin,left] at (0,3) {$_{g_2=2}$};
           \node[thin,left] at (0,4) {$_{g_3=1}$};
    \end{tikzpicture}
      
      \caption{Extension to a facet}
      \label{fig: facet extension}
  \end{figure}
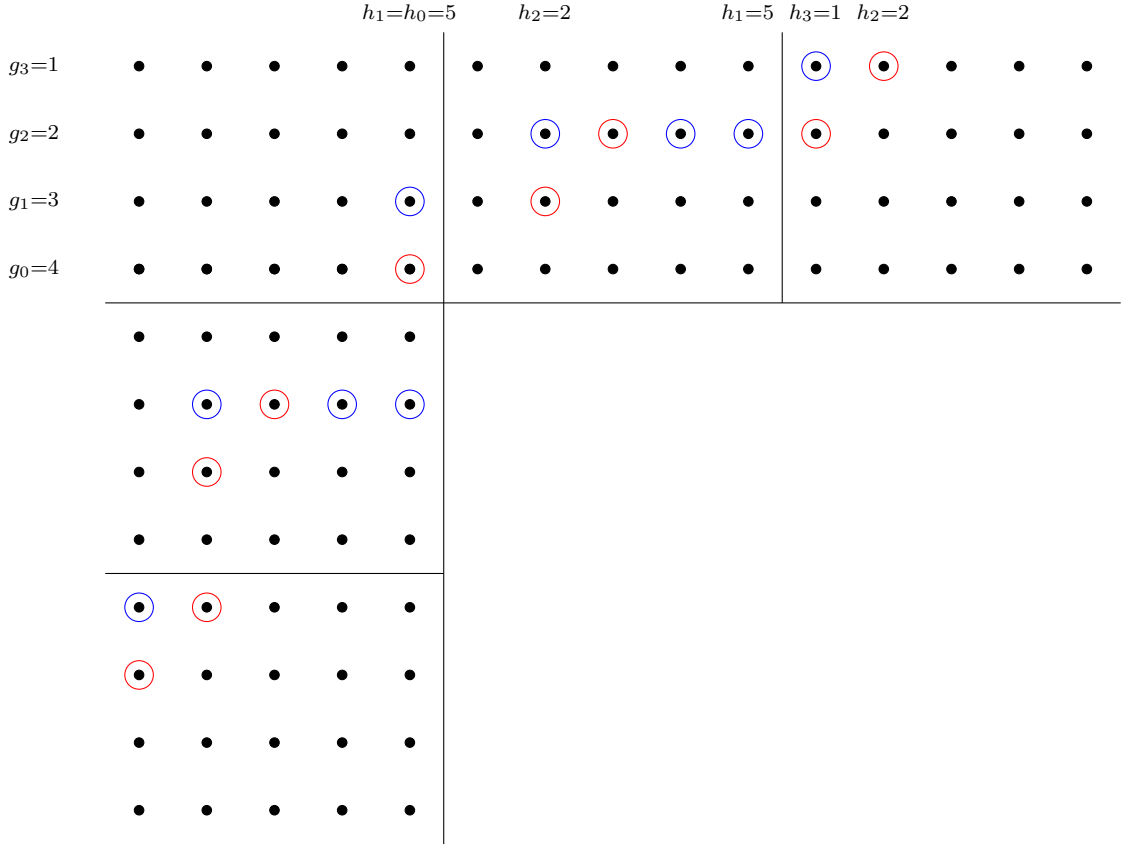
\end{example}

Note that if one considers a Hibi order $\tau$, then the simplicial complex associated to $K[L]$, where $L=\J(P)$, is known and very easily seen to be the order complex of $L$, that is, the simplicial complex of all the chains of $L$. In our case $L=L_{mnr}=\J(P_{mnr})$, and we have seen in the proof of \Cref{thm: invariants}(3) that the facets of this complex (which are the maximal chains of $L_{mnr}$) are in bijection with $S(M_{mnr})$, where $M_{mnr}=\{M^{m-1}, N^{n-1}, R^{r-1}\}$. This implies that $|\F(\Delta)|=|S(M_{mnr})|$, hence there is a set-theoretic bijection between the elements of $\F(\Delta)$ and words on this multiset $M_{mnr}$. This leads to the following question: how can one naturally encode the facets of $\Delta$ with words in $M_{mnr}$? 

\begin{proposition}
\label{NMR}
There is a bijection between $\mathcal{F}(\Delta)$ and the set of the words on the multiset $\{M^{m-1}, N^{n-1}, R^{r-1}\}$.
\end{proposition}
\begin{proof}
First note that, in a facet, any path $\{(g_1,h_1), \ldots, (g_t, h_t) \}$ from $(g_1, h_1)$ to $(g_t, h_t)$ is inside the subregion of $P$ corresponding to a matrix $X_i$ for a certain $i$. 
Moreover it can be written as a unique word on the alphabet $\{M, N\}$ by identifying the step from $(g_s, h_s)$ to $(g_{s+1}, h_{s+1})$ with $N$ if $(g_{s+1}, h_{s+1}) - (g_s, h_s) =  (0,1)$ and with $M$ if $(g_{s+1}, h_{s+1}) - (g_s, h_s) =  (-1,0)$. 

We can thus identify a union of paths as defined in the theorem with a word on $\{M^{m-1}, N^{n-1}, R^{r-1}\}$ by adding an $R$ between paths. 

To see that this process is reversible, suppose we begin with a word $w$ on the multiset $\{M^{m-1}, N^{n-1}, R^{r-1}\}$.  Define subwords $w_1, \dots, w_r$ on the alphabet $\{M,N\}$ so that $w_1$ is the word before the first occurrence of $R$ in $w$,  $w_k$ is the word between the $(k-1)$st and $k$th occurrences of $R$ in $w$ and $w_r$ is the word after the last occurrence of $R$ in $w$.  Note that $w_k$ may be empty for any $k$.  

For each $1 \leq k \leq r-1$, define 
\[
h_k = 1+(\mbox{the number of $N$'s which occur after the $k$th $R$ in $w$}).
\]
Similarly define 
\[
g_k = 1 + (\mbox{the number of $M$'s which occur after the $k$th $R$ in $w$)}.
\]
 Then $n=h_0\geq h_1\geq h_2\geq\ldots \geq h_{r-1}\geq h_r= 1$ and $m=g_0\geq g_1\geq g_2\geq \ldots \geq g_{r-1}\geq g_r= 1$. Further, the word $w_k$ defines a unique path $(g_{k-1}, (k-1)n+h_{k}) \longrightarrow (g_{k},(k-1)n+h_{k-1})$ for $1\leq k \leq r$.  

\end{proof}

\begin{example}
Let $m=4$, $n=5$, and $r=3$. Then by the proposition above the facets of $\Delta$ are in bijection with the words in $\{M^{3},N^{4}, R^{2}\}$. 
Consider the facet $F$ of \Cref{ex:complete-faces}, that is the union of three paths  $(4,5)\longrightarrow (3,5)$, $(3,7)\longrightarrow (2,10)$, and $(2,11)\longrightarrow (1,12)$ as in Figure 1. We construct the word associated to $F$ following the proof of the proposition. 
The first path contains $(4,5)$ and $(3,5)$, and this gives $(3,5)-(4,5)=(-1,0)$, hence the first path translates as $M$.\\
The second path contains $(3,7)$, $(2,7)$, $(2,8)$, $(2,9)$, $(2,10)$, and this gives $(2,7)-(3,7)=(-1,0)$, $(2,8)-(2,7)=(0,1)$, $(2,9)-(2,8)=(0,1)$, $(2,10)-(2,9)=(0,1)$, hence the second path translates to $MNNN$.\\
The last path contains $(2,11)$, $(1,11)$, $(1,12)$, and this gives $(1,11)-(2,11)=(-1,0)$, $(1,12)-(1,11)=(0,1)$, hence the last path translates as $MN$.\\
By adding the three $R$'s in between, we have the code word $MRMNNNRMN$.\\

We also describe the opposite procedure: given the word $MRMNNNRMN$, one finds the corresponding facet (that we already know to be $F$). Since $R$ stands between matrices in the first one there are $2$ points, in the second one  $5$ points, and in the last one $3$ points.
%Hence the facet will contain $10$ points, as it follows  $\{P_1,\ldots,P_{10}\}$ points. 
We want to construct the paths that constitute the facet. 
By the proof of \Cref{NMR} one has $$h_0=5,\quad h_1=5,\quad h_2=2, \quad h_3=1$$
and 
$$g_0=4,\quad g_1=3,\quad g_2=2, \quad g_3=1.$$

By \Cref{facets} our facet is a union of paths of the  form

\[
  	(g_{k-1}, (k-1)n+h_{k}) \longrightarrow (g_{k},(k-1)n+h_{k-1}) \mbox{ for $1\leq k\leq r$.}
\]  
In our case $r=3$ and $n=5$, and we have computed all the values of $h_i$ and $g_i$.

For $k=1$ we get $(g_0, 0\cdot5+h_1)\longrightarrow (g_1, 0\cdot5+h_0)=(4,5)\longrightarrow (3,5)$.

For $k=2$ we get $(g_1, 1\cdot5+h_2)\longrightarrow (g_2, 1\cdot5+h_1)=(3,7)\longrightarrow (2,10)$.

For $k=3$ we get $(g_2, 2\cdot5+h_3)\longrightarrow (g_3, 2\cdot5+h_2)=(2,11)\longrightarrow (1,12)$.

Thus $F$ is the union of the paths $(4,5)\longrightarrow (3,5)$, $(3,7)\longrightarrow (2,10)$, and $(2,11)\longrightarrow (1,12)$. Next, to find all the points, we use the steps. For instance, since we have $M$ after the first $R$, and we know the initial point of the second path which is $(3,7)$, we can find the next point $P=(g_k,h_k)$. The difference $(g_k,h_k)-(3,7)=(-1,0)$ hence we have $P=(2,7)$. Proceeding this way, one can find all the points in the facet.
\end{example}

In the following example, we construct the set $\mathcal{F}(\Delta)$ for fixed values of $m,n,r$.
\begin{example}
Let $m=2$, $n=2$, $r=3$. Then $|\mathcal{F}(\Delta)|=12$ (see proof of \Cref{thm: invariants}) let $\mathcal{F}(\Delta)=\{F_1,\ldots,F_{12}\}$. According to \Cref{facets}, every facet is the union of three paths, associated to the integers $ 2=g_0\geq g_1 \geq g_2\geq g_3=1$ and $ 2=h_0\geq h_1 \geq h_2\geq h_3=1$. Note that two different facets can be obtained with the same $h_k$'s and $g_k$'s.

We give the list of all the facets and we represent them in $P$ (see Figure 2), with also the associated code words obtained as in the proof of \Cref{NMR}. 
\begin{itemize}
    \item[{$F_1,F_2$ :}]  $g_1=1$, $g_2=1$, $h_1=1$, $h_2=1$\\
    $(2,1)\rightarrow (1,2)$, $(1,3)\rightarrow (1,3)$, $(1,5)\rightarrow (1,5)$ 
    \smallskip
    
    \item[{$F_3$} :] $g_1=2$, $g_2=1$, $h_1=1$, $h_2=1$ \\
     $(2,1)\rightarrow (2,2)$, $(2,3)\rightarrow (1,3)$, $(1,5)\rightarrow (1,5)$
      \smallskip
      
     \item[{$F_4$} :] $g_1=2$, $g_2=2$, $h_1=1$, $h_2=1$\\
      $(2,1)\rightarrow (2,2)$, $(2,3)\rightarrow (2,3)$,  $(2,5)\rightarrow (1,5)$
       \smallskip
       
     \item[{$F_5$} :] $g_1=2$, $g_2=2$, $h_1=2$, $h_2=1$\\
      $(2,2)\rightarrow (2,2)$,  $(2,3)\rightarrow (2,4)$,  $(2,5)\rightarrow (1,5)$ 
       \smallskip

     \item[{$F_6,F_7$} :] $g_1=2$, $g_2=1$, $h_1=2$, $h_2=1$\\
       $(2,2)\rightarrow (2,2)$,  $(2,3)\rightarrow (1,4)$,  $(1,5)\rightarrow (1,5)$
        \smallskip
        
      \item[{$F_8$} :] $g_1=1$, $g_2=1$, $h_1=2$, $h_2=1$\\
       $(2,2)\rightarrow (1,2)$,  $(1,3)\rightarrow (1,4)$,  $(1,5)\rightarrow (1,5)$
        \smallskip
        
      \item[{$F_9$} :] $g_1=1$, $g_2=1$, $h_1=2$, $h_2=2$\\
        $(2,2)\rightarrow (1,2)$,  $(1,4)\rightarrow (1,4)$,  $(1,5)\rightarrow (1,6)$
         \smallskip
         
       \item[{$F_{10}$} :] $g_1=2$, $g_2=1$, $h_1=2$, $h_2=2$\\
       $(2,2)\rightarrow (2,2)$,  $(2,4)\rightarrow (1,4)$,  $(1,5)\rightarrow (1,6)$
        \smallskip
        
       \item[{$F_{11},F_{12}$} :] $g_1=2$, $g_2=2$, $h_1=2$, $h_2=2$\\
     $(2,2)\rightarrow (2,2)$,  $(2,4)\rightarrow (2,4)$, $(2,5)\rightarrow (1,6)$
 \end{itemize}

    \begin{figure}[H]
        \centering
        \includegraphics[width=1\linewidth]{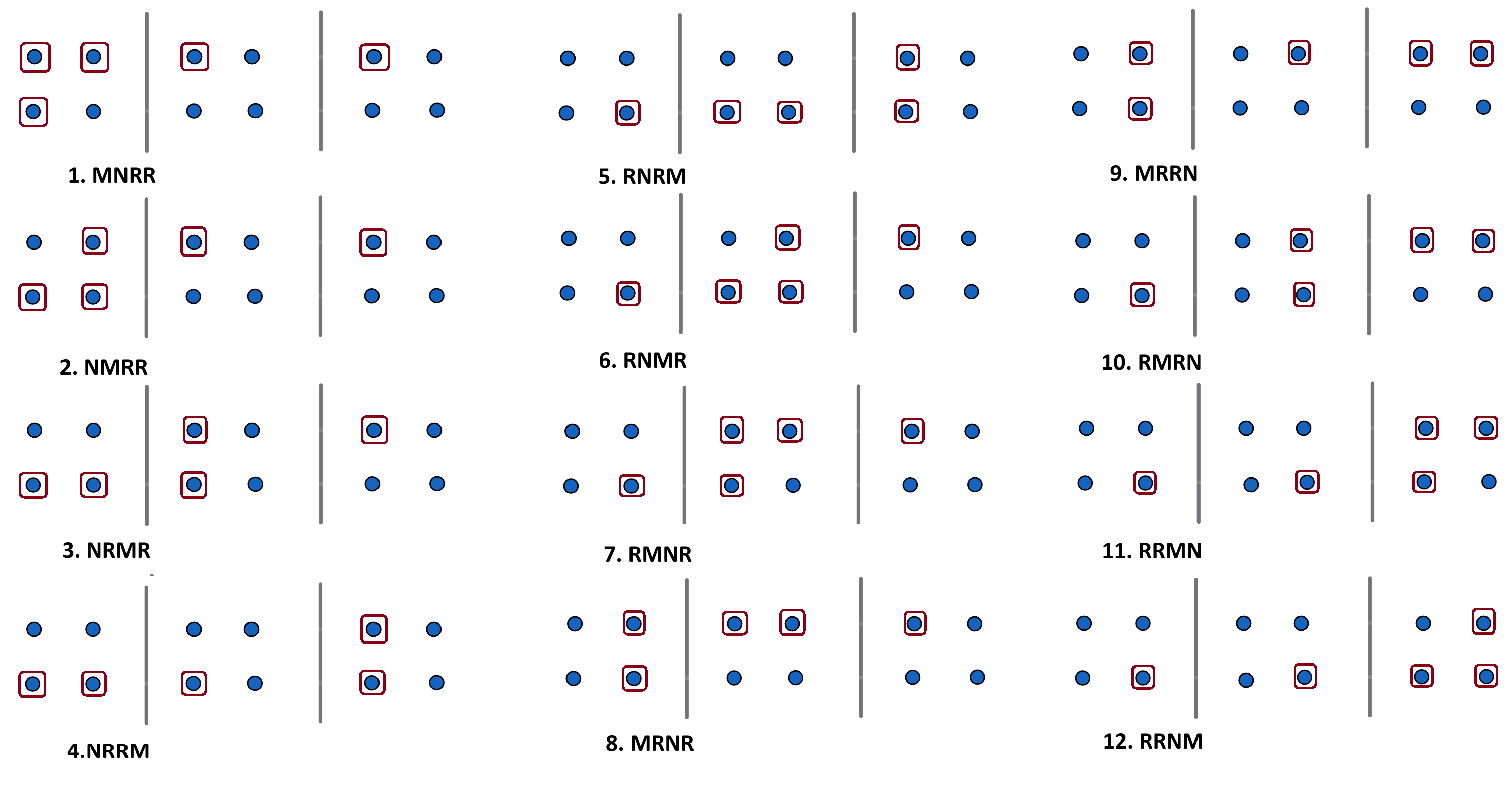}
        \caption{All the facets of $\Delta$ for $m=2$, $n=2$, $r=3$.}
        \label{fig:enter-label}
    \end{figure}
\end{example}

Note that the the facets in this example have been ordered that so that $F_{\ell}$ differs from $F_{\ell+1}$ by exchanging exactly one pair of adjacent letters in the corresponding word.  In this example this is a shelling order of the simplicial complex. This leads us the to the following two questions for further work.

\begin{question}
Does any order of the facets of $\Delta$ in which subsequent facets differ by an exchange of adjacent letters in the corresponding word constitute a shelling order?
\end{question}

\begin{question}
Can the formulation of the facets of $\Delta$ as words on an alphabet be extended to the case of larger minors, and if so is there a similar shelling order?
\end{question}

 \subsection*{Acknowledgments} 
 
 Our work started at the 2024 workshop ``Women in Commutative Algebra III'' hosted by Casa Matem\'atica Oaxaca. We thank the organizers of this workshop for bringing our team together. We acknowledge the excellent working conditions provided by CMO-BIRS.

 The second author was partially supported by the MIUR Excellence Department Project CUP D33C23001110001, and by INdAM-GNSAGA. 

 The third author was supported by a fellowship from the Wenner-Gren Foundations (grant WGF2022-0052). She would also like to thank Fabian Burghart and Per Alexandersson for helping her figure out some of the combinatorics in this paper.

  The fourth author was partially supported by the Scientific and Technological Research Council of Turkey - TÜBİTAK (Grant No: 122F128) and WICA III.

The fifth author would like to thank the SERB ITS scheme (File No. ITS/2024/001271) for providing the travel support to participate in WICAIII.

\end{document}